\documentclass[11pt,preprint]{imsart}
\usepackage{amsmath,amsthm,verbatim,amssymb,amsfonts,amscd, graphicx}
\usepackage{amssymb,bm}
\usepackage{mathrsfs}
\usepackage{hyperref}
\hypersetup{
    colorlinks,
    citecolor=black,
    filecolor=black,
    linkcolor=blue,
    urlcolor=black
}
\usepackage{constants}
\usepackage{enumerate}
\newcommand{\rom}[1]{\uppercase\expandafter{\romannumeral #1\relax}}
\newcommand{\beas}{\begin{eqnarray*}}
\newcommand{\enas}{\end{eqnarray*}}
\newcommand{\bea}{\begin{eqnarray}}
\newcommand{\ena}{\end{eqnarray}}
\newcommand{\bms}{\begin{multline*}}
\newcommand{\ems}{\end{multline*}}
\newcommand{\bels}{\begin{align*}}
\newcommand{\enls}{\end{align*}}
\newcommand{\bel}{\begin{align}}
\newcommand{\enl}{\end{align}}

\newcommand{\ignore}[1]{}
\newcommand{\tr}{\mbox{tr}}

\newtheorem{theorem}{Theorem}[section]
\newtheorem{corollary}{Corollary}[section]

\newtheorem{remark}{Remark}[section]
\newtheorem{lemma}{Lemma}[section]

\def\blfootnote{\xdef\@thefnmark{}\@footnotetext}

\newcommand{\expect}[1]{\mathbb{E}{\l[#1\r]}}

\def\l{\left}
\def\r{\right}
\newcommand{\m}{\mathcal}
\newcommand{\mb}{\mathbb}
\newcommand\argmin{\mathop{\mbox{argmin}}}
\newcommand{\sign}{\mbox{sign}}

\newcommand{\proj}{\mbox{{\rm Proj}}}

\newcommand{\rank}{\mbox{{\rm rank}}}
\newcommand{\eps}{\varepsilon}
\newcommand{\card}{\mathrm{card}}
\newcommand{\wh}{\widehat}

\newcommand{\med}[1]{\mbox{med}\left(#1\right)}

\newcommand{\pr}[1]{\mathbb{P}{\left(#1\right)}}
\newcommand{\dotp}[2]{\left\langle#1,#2\right\rangle}

\newcommand{\mx}{\mbox{\footnotesize{max}\,}}
\newcommand{\mn}{\mbox{\footnotesize{min}\,}}

\begin{document}
\begin{frontmatter}
\title{Estimation of the covariance structure of heavy-tailed distributions}
\runtitle{Estimation of the covariance structure of heavy-tailed distributions}
%

\begin{aug}
\author{\fnms{Stanislav} \snm{Minsker}\thanksref{t2}\ead[label=e2]{minsker@usc.edu}}
\and
\author{\fnms{Xiaohan} \snm{Wei}\thanksref{t1}\ead[label=e3]{xiaohanw@usc.edu}}
 \thankstext{t2}{Department of Mathematics, University of Southern California}
  \thankstext{t1}{Department of Electrical Engineering, University of Southern California}
\runauthor{S. Minsker, X. Wei}

\affiliation{University of Southern California}
\printead{e2,e3}
\end{aug}

\maketitle

\begin{abstract}
We propose and analyze a new estimator of the covariance matrix that admits strong theoretical guarantees under weak assumptions on the underlying distribution, such as existence of moments of only low order. 
While estimation of covariance matrices corresponding to sub-Gaussian distributions is well-understood, much less in known in the case of heavy-tailed data. 
As K. Balasubramanian and M. Yuan write \footnote{\cite{balasubramanian2016discussion}}, ``data from real-world experiments oftentimes tend to be corrupted with outliers and/or exhibit heavy tails. In such cases, it is not clear that those covariance matrix estimators .. remain optimal'' and ``..what are the other possible strategies to deal with heavy tailed distributions warrant further studies.'' 
We make a step towards answering this question and prove tight deviation inequalities for the proposed estimator that depend only on the parameters controlling the ``intrinsic dimension'' associated to the covariance matrix (as opposed to the dimension of the ambient space); in particular, our results are applicable in the case of high-dimensional observations.

\end{abstract}
\end{frontmatter}

\section{Introduction}

Estimation of the covariance matrix is one of the fundamental problems in data analysis: many important statistical tools, such as Principal Component Analysis(PCA) \cite{hotelling1933analysis} and regression analysis, involve covariance estimation as a crucial step. 
For instance, PCA has immediate applications to nonlinear dimension reduction and manifold learning techniques \cite{allard2012multi}, genetics \cite{novembre2008genes}, computational biology \cite{alter2000singular}, among many others. 

However, assumptions underlying the theoretical analysis of most existing estimators, such as various modifications of the sample covariance matrix, are often restrictive and do not hold for real-world scenarios. 
Usually, such estimators rely on heuristic (and often bias-producing) data preprocessing, such as outlier removal.  
To eliminate such preprocessing step from the equation, one has to develop a class of new statistical estimators that admit strong performance guarantees, such as exponentially tight concentration around the unknown parameter of interest, under weak assumptions on the underlying distribution, such as existence of moments of only low order. 
In particular, such heavy-tailed distributions serve as a viable model for data corrupted with outliers -- an almost inevitable scenario for applications. 

We make a step towards solving this problem: using tools from the random matrix theory, we will develop a class of \textit{robust} estimators that are numerically tractable and are supported by strong theoretical evidence under much weaker conditions than currently available analogues. The term ``robustness'' refers to the fact that our estimators admit provably good performance even when the underlying distribution is heavy-tailed.

\subsection{Notation and organization of the paper}

Given $A\in \mb R^{d_1\times d_2}$, let $A^T\in \mb R^{d_2\times d_1}$ be transpose of $A$. 
If $A$ is symmetric, we will write $\lambda_{\mx}(A)$ and $\lambda_{\mn}(A)$ for the largest and smallest eigenvalues of $A$. 
Next, we will introduce the matrix norms used in the paper. 
Everywhere below, $\|\cdot\|$ stands for the operator norm $\|A\|:=\sqrt{\lambda_{\mx}(A^T A)}$. 
If $d_1=d_2=d$, we denote by $\tr A$ the trace of $A$.
 For $A\in \mb R^{d_1\times d_2}$, the nuclear norm $\|\cdot\|_1$ is defined as 
$\|A\|_1=\tr(\sqrt{A^T A})$, where $\sqrt{A^T A}$ is a nonnegative definite matrix such that $(\sqrt{A^T A})^2=A^T A$. 
The Frobenius (or Hilbert-Schmidt) norm is $\|A\|_{\mathrm{F}}=\sqrt{\tr(A^T A)}$, and the associated inner product is 
$\dotp{A_1}{A_2}=\tr(A_1^\ast A_2)$. 
For $z\in \mb R^d$, $\l\| z \r\|_2$ stands for the usual Euclidean norm of $z$. 
Let $A$, $B$ be two self-adjoint matrices. We will write $A\succeq B \ (\text{or }A\succ B)$ iff $A-B$ is nonnegative (or positive) definite.
For $a,b\in \mb R$, we set $a\vee b:=\max(a,b)$ and $a\wedge b:=\min(a,b)$. 
We will also use the standard Big-O and little-o notation when necessary.  

Finally, we give a definition of a matrix function. 
Let $f$ be a real-valued function defined on an interval $\mb T\subseteq \mb R$, and let $A\in \mb R^{d\times d}$ be a symmetric matrix with the eigenvalue decomposition 
$A=U\Lambda U^\ast$ such that $\lambda_j(A)\in \mb T,\ j=1,\ldots,d$. 
We define $f(A)$ as 
$f(A)=Uf(\Lambda) U^\ast$, where 
\[
f(\Lambda)=f\l( \begin{pmatrix}
\lambda_1 & \,  & \,\\
\, & \ddots & \, \\
\, & \, & \lambda_d
\end{pmatrix} \r)
:=\begin{pmatrix}
f(\lambda_1) & \,  & \,\\
\, & \ddots & \, \\
\, & \, & f(\lambda_d)
\end{pmatrix}.
\] 
Few comments about organization of the material in the rest of the paper: section \ref{sec:literature} provides an overview of the related work. Section \ref{sec:main} contains the mains results of the paper. 
The proofs are outlined in section \ref{sec:proofs}; longer technical arguments can be found in the supplementary material. 

\subsection{Problem formulation and overview of the existing work}
\label{sec:literature}

Let $X\in \mb R^d$ be a random vector with 
mean $\mb EX = \mu_0$, covariance matrix $\Sigma_0 = \mb E\l[ (X - \mu_0)(X - \mu_0)^T \r]$, and assume $\mb E \|X - \mu_0\|_2^4<\infty$. 
Let $X_1,\ldots, X_{m}$ be i.i.d. copies of $X$. 
Our goal is to estimate the covariance matrix $\Sigma$ from $X_j, \ j\leq m$. 
This problem and its variations have previously received significant attention by the research community: excellent expository papers by \cite{cai2016} and \cite{fan2016overview} discuss the topic in detail. 
However, strong guarantees for the best known estimators hold (with few exceptions mentioned below) under the restrictive assumption that $X$ is either bounded with probability 1 or has sub-Gaussian distribution, meaning that there exists $\sigma>0$ such that for any $v\in \mb R^d$ of unit Euclidean norm, 
\[
\Pr\l(\l|\dotp{v}{X-\mu_0}\r|\geq t \r)\leq 2 e^{-\frac{t^2 \sigma^2}{2}}.
\]
In the discussion accompanying the paper by \cite{cai2016}, \cite{balasubramanian2016discussion} write that ``data from real-world experiments oftentimes tend to be corrupted with outliers and/or exhibit heavy tails.
In such cases, it is not clear that those covariance matrix estimators described in this article remain optimal'' and ``..what are the other possible strategies to deal with heavy tailed distributions warrant further studies.'' 
This motivates our main goal: develop new estimators of the covariance matrix that (i) are computationally tractable and perform well when applied to heavy-tailed data and (ii) admit strong theoretical guarantees (such as exponentially tight concentration around the unknown covariance matrix) under weak assumptions on the underlying distribution. 
Note that, unlike the majority of existing literature, we do not impose \textit{any further conditions} on the moments of $X$, or on the ``shape'' of its distribution, such as elliptical symmetry. 

Robust estimators of covariance and scatter have been studied extensively during the past few decades. 
However, majority of rigorous theoretical results were obtained for the class of elliptically symmetric distributions which is a natural generalization of the Gaussian distribution; we mention just a small subsample among the thousands of published works.  
Notable examples include the Minimum Covariance Determinant estimator and the Minimum Volume Ellipsoid estimator which are discussed in \cite{hubert2008high}, as well Tyler's \cite{tyler1987distribution} M-estimator of scatter. 
Works by \cite{fan2016overview,wegkamp2016adaptive, han2016eca} exploit the connection between Kendall's tau and Pearson's correlation coefficient \cite{fang1990symmetric} in the context of elliptical distributions to obtain robust estimators of correlation matrices. 
Interesting results for shrinkage-type estimators have been obtained by \cite{ledoit2004well,ledoit2012nonlinear}. 
In a recent work, \cite{chen2015robust} study Huber's $\eps$-contamination model which assumes that the data is generated from the distribution of the form $(1-\eps)F + \eps Q$, where $Q$ is an arbitrary distribution of ``outliers'' and $F$ is an elliptical distribution of ``inliers'', and propose novel estimator based on the notion of ``matrix depth'' which is related to Tukey's depth function \cite{tukey1975mathematics}; a related class of problems has been studies by \cite{diakonikolas2016robust}.  
The main difference of the approach investigated in this paper is the ability to handle a much wider class of distributions that are not elliptically symmetric and only satisfy weak moment assumptions. 
Recent papers by \cite{catoni2016pac}, \cite{giulini2015pac}, \cite{fan2016eigenvector,fan2017estimation,fan2017robust} and \cite{minsker2016sub} are closest in spirit to this direction. 
For instance, \cite{catoni2016pac} constructs a robust estimator of the Gram matrix of a random vector $Z\in \mb R^d$ (as well as its covariance matrix) via estimating the quadratic form $\mb E \dotp{Z}{u}^2$ uniformly over all $\|u\|_2=1$. 
However, the bounds are obtained under conditions more stringent than those required by our framework, and resulting estimators are difficult to evaluate in applications even for data of moderate dimension. 
\cite{fan2016eigenvector} obtain bounds in norms other than the operator norm which the focus of the present paper. 
\cite{minsker2016sub} and \cite{fan2016robust} use adaptive truncation arguments to construct robust estimators of the covariance matrix. 
However, their results are only applicable to the situation when the data is centered (that is, $\mu_0=0$). 
In the robust estimation framework, rigorous extension of the arguments to the case of non-centered high-dimensional observations is non-trivial and requires new tools, especially if one wants to avoid statistically inefficient procedures such as sample splitting. 
We formulate and prove such extensions in this paper.



\section{Main results}
\label{sec:main}

Definition of our estimator has its roots in the technique proposed by \cite{catoni2012challenging}. 
Let
\begin{align}
\label{eq:psi2}
\psi(x) = \l( |x|\wedge 1 \r)\sign(x)
\end{align}
be the usual truncation function. 
As before, let $X_1,\ldots,X_m$ be i.i.d. copies of $X$, and assume that $\wh \mu$ is a suitable estimator of the mean $\mu_0$ from these samples, to be specified later. 
We define $\wh\Sigma$ as 
\begin{align}
\label{eq:rob-cov}
\wh\Sigma := \frac{1}{m\theta}\sum_{i=1}^m \psi\l( \theta(X_i - \wh\mu)(X_i - \wh\mu)^T \r),  
\end{align}
where $\theta\simeq m^{-1/2}$ is small (the exact value will be given later). 
It easily follows from the definition of the matrix function that 
\[
\wh\Sigma = \frac{1}{m\theta}\sum_{i=1}^m \frac{(X_i - \wh\mu)(X_i - \wh\mu)^T}{\l\| X_i - \wh\mu\r\|_2^2} \psi\l( \theta \l\| X_i - \wh\mu \r\|_2^2 \r),
\]
hence it is easily computable. 
Note that $\psi(x)= x$ in the neighborhood of $0$; it implies that whenever all random variables $\theta \l\| X_i - \wh\mu \r\|_2^2, \ 1\leq i\leq m$ are ``small'' (say, bounded above by $1$) and $\hat\mu$ is the sample mean, $\wh \Sigma$ is close to the usual sample covariance estimator. 
On the other hand, $\psi$ ``truncates'' $\l\| X_i - \wh\mu \r\|_2^2$ on level $\simeq \sqrt{m}$, thus limiting the effect of outliers. 
Our results (formally stated below, see Theorem \ref{th:lepski}) imply that for an appropriate choice of $\theta=\theta(t,m,\sigma)$,
\[
\l\| \wh\Sigma - \Sigma_0 \r\| \leq C_0\sigma_0\sqrt{\frac{\beta}{m}} 
\]
with probability $\geq 1 - d e^{-\beta}$ for some positive constant $C_0$, where 
\[
\sigma_0^2 := \l\| \mb E \l\| X - \mu_0\r\|_2^2 (X - \mu_0)(X - \mu_0)^T \r\|
\] 
is the "matrix variance".

\subsection{Robust mean estimation}
\label{ssec:mean}

There are several ways to construct a suitable estimator of the mean $\mu_0$. 
We present the one obtained via the ``median-of-means'' approach. 
Let $x_1,\ldots,x_k\in \mb R^d$. Recall that the \textit{geometric median} of $x_1,\ldots,x_k$ is defined as
\[
\med{x_1,\ldots,x_k}:=\argmin\limits_{z\in \mb R^d}\sum_{j=1}^k \l\|z- x_j \r\|_2.
\]
Let $1<\beta<\infty$ be the confidence parameter, and set $k=\Big\lfloor 3.5 \beta\Big\rfloor+1$; we will assume that $k\leq \frac{m}{2}$.
Divide the sample $X_1,\ldots, X_m$ into $k$ disjoint groups $G_1,\ldots, G_k$ of size $\Big\lfloor \frac{m}{k}\Big\rfloor$ each, and define 
\begin{align}
\label{eq:median_mean}
\nonumber
\hat\mu_j&:=\frac{1}{|G_j|}\sum_{i\in G_j}X_i, \ j=1\ldots k,\\
\hat\mu&:=\med{\hat\mu_1,\ldots,\hat\mu_k}.
\end{align}
It then follows from Corollary 4.1 in \cite{minsker2013geometric} that
\begin{align}
\label{eq:deviation1}
&
\Pr\Big(\l\| \hat\mu-\mu \r\|_2 \geq 11\sqrt{\frac{\tr(\Sigma_0)(\beta+1)}{m}}\Big)\leq e^{-\beta}.
\end{align}

\subsection{Robust covariance estimation}

Let $\wh\Sigma$ be the estimator defined in \eqref{eq:rob-cov} with $\wh\mu$ being the ``median-of-means'' estimator \eqref{eq:median_mean}. 
Then $\wh \Sigma$ admits the following performance guarantees:
\begin{lemma}
\label{lemma:main}
Assume that $\sigma \geq \sigma_0$, and set $\theta=\frac{1}{\sigma}\sqrt{\frac{\beta}{m}}$.
Moreover, let $\overline{d}:=\sigma_0^2/\|\Sigma_0\|^2$, and suppose that $m\geq C\overline{d}\beta$, where $C>0$ is an absolute constant. Then
\begin{equation}
\label{simple-bound}
\left\| \wh\Sigma - \Sigma_0\right\| \leq 3\sigma\sqrt{\frac{\beta}{m}}
\end{equation}
with probability at least $1-5d e^{-\beta}$.
\end{lemma}
\begin{remark}
The quantity $\bar d$ is a measure of ``intrinsic dimension'' akin to the ``effective rank'' $r=\frac{\tr\l(\Sigma_0\r)}{\|\Sigma_0\|}$; see Lemma \ref{effective-rank-bound} below for more details. 
Moreover, note that the claim of Lemma \ref{lemma:main} holds for any $\sigma\geq\sigma_0$, rather than just for $\sigma=\sigma_0$; this ``degree of freedom'' allows construction of adaptive estimators, as it is shown below.  
\end{remark}


The statement above suggests that one has to know the value of (or a tight upper bound on) the ``matrix variance'' 
$\sigma_0^2$ in order to obtain a good estimator $\widehat\Sigma$. 
More often than not, such information is unavailable. 
To make the estimator completely data-dependent, we will use Lepski's method \cite{lepskii1992asymptotically}. 
To this end, assume that $\sigma_{\mn}, \ \sigma_{\mx}$ are ``crude'' preliminary bounds such that 
\[
\sigma_{\mn}\leq \sigma_0 \leq \sigma_{\mx}.
\]
Usually, $\sigma_{\mn}$ and $\sigma_{\mx}$ do not need to be precise, and can potentially differ from $\sigma_0$ by several orders of magnitude. Set 
\[
\sigma_j := \sigma_{\mn} 2^j \text{ and }
\m J=\l\{ j\in \mb Z: \  \sigma_{\mn} \leq \sigma_j  < 2\sigma_{\mx} \r\}.
\]
Note that the cardinality of $J$ satisfies $\card(\m J)\leq 1+\log_2(\sigma_{\mx}/\sigma_{\mn})$. 
For each $j\in \m J$, define $\theta_j:=\theta(j,\beta) = \frac{1}{\sigma_j} \sqrt{\frac{\beta}{m}}$. 
Define
\[
\wh\Sigma_{m,j}=\frac{1}{m\theta_j}\sum_{i=1}^m \psi\l( \theta_j (X_i-\wh\mu)(X_i-\wh\mu)^T \r).
\]
Finally, set 
\begin{align}
\label{eq:lepski}
j_\ast:=\min\l\{ j\in \m J: \forall k>j \text{ s.t. } k\in \m J,\ \l\|  \wh\Sigma_{m,k} - \wh\Sigma_{m,j} \r\|\leq 6 \sigma_{k} \sqrt{\frac{\beta}{m}}  \r\}
\end{align}
and $\wh\Sigma_\ast:=\wh\Sigma_{m,j_\ast}$. 
Note that the estimator $\wh\Sigma_\ast$ depends only on $X_1,\ldots,X_m$, as well as $\sigma_{\mn}, \ \sigma_{\mx}$. 
Our main result is the following statement regarding the performance of the data-dependent estimator $\wh\Sigma_\ast$:
\begin{theorem}
\label{th:lepski}
Suppose $m\geq C\overline{d}\beta$, then,
the following inequality holds with probability at least $1 - 5d \log_2\l(\frac{2\sigma_{\mx}}{\sigma_{\mn}}\r) e^{-\beta}$:
\[
\l\| \wh\Sigma_\ast - \Sigma_0 \r\| \leq 18\sigma_0 \sqrt{\frac{\beta}{m}}.
\]
\end{theorem}

An immediate corollary of Theorem \ref{th:lepski} is the quantitative result for the performance of PCA based on the estimator 
$\wh\Sigma_\ast$. 
Let $\proj_k$ be the orthogonal projector on a subspace corresponding to the $k$ largest positive eigenvalues $\lambda_1,\ldots,\lambda_k$ of $\Sigma_0$ (here, we assume for simplicity that all the eigenvalues are distinct), and $\wh{\proj_k}$ -- the orthogonal projector of the same rank as $\proj_k$ corresponding to the $k$ largest eigenvalues of $\wh\Sigma_\ast$. 
The following bound follows from the Davis-Kahan perturbation theorem \cite{davis1970rotation}, more specifically, its version due to \cite[][Theorem 3 ]{Zwald2006On-the-Converge00}. 
\begin{corollary}
\label{cor:PCA}
Let $\Delta_k=\lambda_k - \lambda_{k+1}$, and assume that $\Delta_k\geq 72\sigma_0 \sqrt{\frac{\beta}{m}}$. 
Then  
\[
\big\|\widehat{\proj_k}-\proj_k\big\| \leq \frac{36}{\Delta_k}\sigma_0 \sqrt{\frac{\beta}{m}} 
\]
with probability $\geq 1 - 5d \log_2\l(\frac{2\sigma_{\mx}}{\sigma_{\mn}}\r) e^{-\beta}$. 
\end{corollary}
It is worth comparing the bound of Lemma \ref{lemma:main} and Theorem \ref{th:lepski} above to results of the paper by \cite{fan2016robust}, which constructs a covariance estimator $\widehat{\Sigma}_m'$ under the assumption that the random vector $X$ is centered, and $\sup_{\mathbf v\in \mb R^d: \|\mathbf{v}\|_2\leq1}\expect{|\langle\mathbf{v},X\rangle|^4}=B<\infty$. 
More specifically, $\widehat{\Sigma}_m'$ satisfies the inequality 
\begin{align}
\label{eq:fan}
\pr{\l\| \widehat{\Sigma}_m'-\Sigma_0 \r\| \geq \sqrt{\frac{C_1\beta Bd}{m}} } \leq de^{-\beta},
\end{align}
where $C_1>0$ is an absolute constant. 
The main difference between \eqref{eq:fan} and the bounds of Lemma \ref{lemma:main} and Theorem \ref{th:lepski} is that the latter are expressed in terms of $\sigma_0^2$, while the former is in terms of $B$. 
The following lemma demonstrates that our bounds are at least as good:
\begin{lemma}
\label{bound-on-sigma}
Suppose that $\mb E X = 0$ and $\sup_{\mathbf v\in \mb R^d:\|\mathbf{v}\|_2\leq 1}\expect{|\langle\mathbf{v},X\rangle|^4}=B<\infty$. 
Then $Bd\geq\sigma_0^2$. 
\end{lemma}
It follows from the above lemma that $\overline{d}=\sigma_0^2/\|\Sigma_0\|^2\lesssim d$.
Hence, By Theorem \ref{th:lepski}, the error rate of estimator $\wh \Sigma_\ast$ is bounded above by $\mathcal{O}(\sqrt{d/m})$ if $m\gtrsim d$. 
It has been shown (for example, see \cite{lounici2014high}) that the minimax lower bound of covariance estimation is of order 
$\Omega(\sqrt{d/m})$. 
Hence, the bounds of \cite{fan2016robust} as well as our results imply correct order of the error. 
That being said, the ``intrinsic dimension'' $\bar d$ reflects the structure of the covariance matrix and can potentially be much smaller than $d$, as it is shown in the next section.

\subsection{Bounds in terms of intrinsic dimension}

In this section, we show that under a slightly stronger assumption on the fourth moment of the random vector $X$, the bound $\mathcal{O}(\sqrt{d/m})$ is suboptimal, while our estimator can achieve a much better rate in terms of the ``intrinsic dimension'' associated to the covariance matrix. 
This makes our estimator useful in applications involving high-dimensional covariance estimation, such as PCA.
Assume the following uniform bound on the \textit{kurtosis} of linear forms $\langle Z,v\rangle$:
\begin{equation}
\label{kurtosis}
\sup_{\|\mathbf{v}\|_2\leq1}\frac{\sqrt{\mb E \dotp{Z}{\mathbf{v}}^4}}{\mb E \dotp{Z}{\mathbf{v}}^2}=R<\infty.
\end{equation}
The intrinsic dimension of the covariance matrix $\Sigma_0$ can be measured by the \textit{effective rank} defined as 
\[
\mathbf{r}(\Sigma_0)=\frac{\tr(\Sigma_0)}{\|\Sigma_0\|}.
\]
Note that we always have $\mathbf{r}(\Sigma_0)\leq \text{rank}(\Sigma_0)\leq d$, and it some situations 
$\mathbf{r}(\Sigma_0)\ll \text{rank}(\Sigma_0)$, for instance if the covariance matrix is ``approximately low-rank'', meaning that it has many small eigenvalues.
The constant $\sigma_0^2$ is closely related to the effective rank as is shown in the following lemma (the proof of which is included in the supplementary material):
\begin{lemma}
\label{effective-rank-bound}
Suppose that \eqref{kurtosis} holds. 
Then,
\[
\mathbf{r}(\Sigma_0)\|\Sigma_0\|^2\leq \sigma_0^2\leq R^2\mathbf{r}(\Sigma_0)\|\Sigma_0\|^2.
\]
\end{lemma}
As a result, we have $\mathbf{r}(\Sigma_0)\leq\overline{d}\leq R^2\mathbf{r}(\Sigma_0)$.
The following corollary immediately follows from Theorem \ref{th:lepski} and Lemma \ref{effective-rank-bound}:
\begin{corollary}
Suppose that $m\geq C\beta \mathbf{r}(\Sigma_0)$ for an absolute constant $C>0$ and that \eqref{kurtosis} holds. 
Then
\[
\l\| \wh\Sigma_\ast - \Sigma_0 \r\| \leq 18R\|\Sigma_0\| \sqrt{\frac{\mathbf{r}(\Sigma_0)\beta}{m}}
\]
with probability at least 
$1 - 5d \log_2\l(\frac{2\sigma_{\mx}}{\sigma_{\mn}}\r) e^{-\beta}$.
\end{corollary}

\section{Applications: low-rank covariance estimation}

In many data sets encountered in modern applications (for instance, gene expression profiles \cite{saal2007poor}), dimension of the observations, hence the corresponding covariance matrix, is larger than the available sample size. 
However, it is often possible, and natural, to assume that the unknown matrix possesses special structure, such as low rank, thus reducing the ``effective dimension'' of the problem. 
The goal of this section is to present an estimator of the covariance matrix that is ``adaptive'' to the possible low-rank structure; such estimators are well-known and have been previously studied for the bounded and sub-Gaussian observations \cite{lounici2014high}. 
We extend these results to the case of heavy-tailed observations; in particular, we show that the estimator obtained via soft-thresholding applied to the eigenvalues of $\wh\Sigma_\ast$ admits optimal guarantees in the Frobenius (as well as operator) norm.

Let $\wh\Sigma_\ast$ be the estimator defined in the previous section, see equation \eqref{eq:lepski}, and set
\begin{align}
&
\wh \Sigma_\ast^{\tau}=\argmin_{A\in \mb R^{d\times d}}
\l[  \l\| A - \wh \Sigma_\ast \r\|^2_{\mathrm{F}} +\tau \l\| A \r\|_1\r],
\end{align}
where $\tau>0$ controls the amount of penalty. 
It is well-known (e.g., see the proof of Theorem 1 in \cite{lounici2014high}) that 
$\wh \Sigma_{2n}^\tau$ can be written explicitly as 
\[
\wh \Sigma_{\ast}^\tau = \sum_{i=1}^d \max\l(\lambda_i\l(\wh \Sigma_{\ast}\r) -\tau/2, 0\r) v_i(\wh \Sigma_{\ast}) v_i(\wh \Sigma_{\ast})^T,
\]
where $\lambda_i(\wh \Sigma_{\ast})$ and $v_i(\wh \Sigma_{\ast})$ are the eigenvalues and corresponding eigenvectors of $\wh \Sigma_{\ast}$. 
We are ready to state the main result of this section. 
\begin{theorem}
\label{th:covariance}
For any 
$
\tau \geq 36 \sigma_0 \sqrt{\frac{\beta}{m}},
$
\begin{align}
&\label{eq:ex70}
\l\| \wh \Sigma_\ast^\tau - \Sigma_0 \r\|_{\mathrm{F}}^2\leq \inf_{A\in \mb R^{d\times d}} \l[  \l\| A - \Sigma_0 \r\|_{\mathrm{F}}^2 + \frac{(1+\sqrt{2})^2}{8}\tau^2\mathrm{rank}(A)  \r].
\end{align}
with probability $\geq 1 - 5d \log_2\l(\frac{2\sigma_{\mx}}{\sigma_{\mn}}\r) e^{-\beta}$.
\end{theorem}
In particular, if $\rank(\Sigma_0) = r$ and $\tau = 36 \sigma_0 \sqrt{\frac{\beta}{m}}$, we obtain that 
\[
\l\| \wh \Sigma_\ast^\tau - \Sigma_0 \r\|_{\mathrm{F}}^2 \leq 162\,\sigma_0^2 \l(1+\sqrt{2}\r)^2 \frac{\beta r}{m}
\]
with probability $\geq 1 - 5d \log_2\l(\frac{2\sigma_{\mx}}{\sigma_{\mn}}\r) e^{-\beta}$.

\section{Proofs}
\label{sec:proofs}


\subsection{Proof of Lemma \ref{lemma:main}}
\label{ssec:mainproof}
The result is a simple corollary of the following statement.
\begin{lemma}
\label{main:lemma-2}
Set $\theta=\frac{1}{\sigma}\sqrt{\frac{\beta}{m}}$, where $\sigma \geq \sigma_0$ and $m\geq\beta$.
Let $\overline{d}:=\sigma_0^2/\|\Sigma_0\|^2$. 
Then, with probability at least $1-5de^{-\beta}$,
\begin{multline*}
\left\| \wh\Sigma - \Sigma_0\right\|
\leq 2\sigma\sqrt{\frac{\beta}{m}} \\
+C'\|\Sigma_0\|  \l( \sqrt{\frac{\overline{d}\sigma}{\|\Sigma_0\|}}\l(\frac{\beta}{m}\r)^{\frac34} + \frac{\sqrt{\overline{d}}\sigma}{\|\Sigma_0\|}\frac{\beta}{m}
+ \sqrt{\frac{\overline{d}\sigma}{\|\Sigma_0\|}}\l(\frac{\beta}{m}\r)^{\frac54}  
+\overline{d}
\l(\frac{\beta}{m}\r)^{\frac32} + \frac{\overline{d}\beta^2}{m^2} + \overline{d}^{\frac54}\l(\frac{\beta}{m}\r)^{\frac94} \right),
\end{multline*}
where $C'>1$ is an absolute constant.
\end{lemma}
Now, by Corollary \ref{FKG-bound} in the supplement, it follows that 
$\overline{d} = \sigma_0^2/\|\Sigma_0\|^2\geq\tr(\Sigma_0)/\|\Sigma_0\|\geq1$. Thus, 
assuming that the sample size satisfies $m\geq(6C')^4\overline{d}\beta$, then, 
$\overline{d}\beta/m\leq1/(6C')^4<1$, and by some algebraic manipulations
we have that
\begin{equation}\label{need-steps}
\left\| \wh\Sigma - \Sigma_0\right\|
\leq 2\sigma\sqrt{\frac{\beta}{m}} +  \sigma\sqrt{\frac{\beta}{m}}=3\sigma\sqrt{\frac{\beta}{m}}.
\end{equation}
For completeness, a detailed computation is given in the supplement. This
finishes the proof.

\subsection{Proof of Lemma \ref{main:lemma-2}}

Let $B_\beta = 11\sqrt{2\tr(\Sigma_0)\beta/m}$ be the error bound of the robust mean estimator $\wh\mu$ defined in \eqref{eq:median_mean}.
Let $Z_i = X_i - \mu_0$,
$\Sigma_\mu = \expect{(Z_i-\mu)(Z_i-\mu)^T}$, $\forall i=1,2,\cdots,d$, and
\[
\hat{\Sigma}_\mu = \frac{1}{m\theta}\sum_{i=1}^m \frac{(X_i - \mu)(X_i - \mu)^T}{\l\| X_i - \mu\r\|_2^2} \psi\l( \theta \l\| X_i - \mu \r\|_2^2 \r),
\]
for any $\|\mu\|_2\leq B_\beta$. 
We begin by noting that the error can be bounded by the supremum of an empirical process indexed by $\mu$, i.e.
\begin{equation}\label{triangle-inequality}
\left\| \hat{\Sigma} - \Sigma_0 \right\| 
\leq \sup_{\|\mu\|_2\leq B_\beta}\left\| \hat{\Sigma}_\mu - \Sigma_0 \right\|
\leq \sup_{\|\mu\|_2\leq B_\beta}\left\| \hat{\Sigma}_\mu - \Sigma_\mu \right\| 
+ \left\| \Sigma_\mu - \Sigma_0 \right\|
\end{equation}
with probability at least $1-e^{-\beta}$.
We first estimate the second term $\left\| \Sigma_\mu - \Sigma_0 \right\|$. 
For any $\|\mu\|_2\leq B_\beta$,
\begin{multline*}
\left\| \Sigma_\mu - \Sigma_0 \right\| 
= \left\| \expect{(Z_i-\mu)(Z_i-\mu)^T - Z_iZ_i^T} \right\|
= \sup_{\mathbf{v}\in \mb R^d:\|\mathbf{v}\|_2 \leq 1} \left| \expect{\dotp{Z_i-\mu}{\mathbf{v}}^2 - \dotp{Z_i}{\mathbf{v}}^2 } \right|     \\
= (\mu^T\mathbf{v})^2 \leq \|\mu\|_2^2 \leq B_\beta^2 =242 \frac{\tr(\Sigma_0)\beta}{m},
\end{multline*}
with probability at least $1-e^{-\beta}$.
It follows from Corollary \ref{FKG-bound} in the supplement that with the same probability
\begin{equation}
\label{mean-bound}
\left\| \Sigma_\mu - \Sigma_0 \right\| \leq 242\frac{\sigma_0^2\beta}{\|\Sigma_0\|m}
\leq 242\frac{\sigma^2\beta}{\|\Sigma_0\|m} = 242\|\Sigma_0\|\frac{\overline{d}\beta}{m}.
\end{equation}
Our main task is then to bound the first term in \eqref{triangle-inequality}. 
To this end, we rewrite it as a double supremum of an empirical process:
\[
\sup_{\|\mu\|_2\leq B_\beta}\left\| \hat{\Sigma}_\mu - \Sigma_\mu \right\|
= \sup_{\|\mu\|_2\leq B_\beta,\|\mathbf{v}\|_2\leq1} \l|\mathbf{v}^T\l(\hat{\Sigma}_\mu - \Sigma_\mu\r)\mathbf{v}\r|
\]
It remains to estimate the supremum above. 
\begin{lemma}
\label{key-lemma}
Set $\theta=\frac{1}{\sigma}\sqrt{\frac{\beta}{m}}$, where $\sigma \geq \sigma_0$ and $m\geq\beta$.
Let $\overline{d}:=\sigma_0^2/\|\Sigma_0\|^2$. 
Then, with probability at least $1-4de^{-\beta}$,
\begin{multline*}
\sup_{\|\mu\|_2\leq B_\beta,\|\mathbf{v}\|_2\leq1} \l|\mathbf{v}^T\l(\hat{\Sigma}_\mu - \Sigma_\mu\r)\mathbf{v}\r|
\leq 2\sigma\sqrt{\frac{\beta}{m}} \\
+C''\|\Sigma_0\|  \l( \sqrt{\frac{\overline{d}\sigma}{\|\Sigma_0\|}}\l(\frac{\beta}{m}\r)^{\frac34} + \frac{\sqrt{\overline{d}}\sigma}{\|\Sigma_0\|}\frac{\beta}{m}
+ \sqrt{\frac{\overline{d}\sigma}{\|\Sigma_0\|}}\l(\frac{\beta}{m}\r)^{\frac54}  
+\overline{d}
\l(\frac{\beta}{m}\r)^{\frac32} + \frac{\overline{d}\beta^2}{m^2} + \overline{d}^{\frac54}\l(\frac{\beta}{m}\r)^{\frac94} \right),
\end{multline*}
where $C''>1$ is an absolute constant.
\end{lemma}
Note that $\sigma\geq\sigma_0$ by defnition, thus, $\overline{d}\leq\sigma^2/\|\Sigma_0\|^2$.
Combining the above lemma with \eqref{triangle-inequality} and \eqref{mean-bound} finishes the proof.


\subsection{Proof of Theorem \ref{th:lepski}}
\label{ssec:lepskiproof}

Define $\bar j:=\min\l\{  j\in \m J: \ \sigma_j \geq \sigma_0\r\}$, and note that $\sigma_{\bar j}\leq 2\sigma_0$. 
We will demonstrate that $j_\ast \leq \bar j$ with high probability. 
Observe that
\begin{align*}
\Pr\l( j_\ast > \bar j\r)&\leq \Pr\l( \bigcup_{k\in \m J: k>\bar j} \l\{  \l\|  \wh \Sigma_{m,k} - \Sigma_{m,\bar j} \r\| > 6\sigma_k \sqrt{\frac{\beta}{n}} \r\} \r)\\
& 
\leq \Pr\l(  \l\|  \wh\Sigma_{m,\bar j} - \Sigma_0 \r\| > 3\sigma_{\bar j} \sqrt{\frac{\beta}{m}} \r) + 
\sum_{k\in \m J: \ k>\bar j}\Pr\l( \l\| \wh\Sigma_{m,k} - \Sigma_0 \r\| > 3\sigma_k \sqrt{\frac{\beta}{m}}  \r)   \\
&
\leq 5de^{-\beta} + 5d \log_2\l(\frac{\sigma_{\mx}}{\sigma_{\mn}}\r) e^{-\beta},
\end{align*}
where we applied \eqref{simple-bound} to estimate each of the probabilities in the sum under the assumption that the number of samples $m\geq C\overline{d}\beta$ and $\sigma_k\geq\sigma_{\bar j}\geq\sigma_0$. 
It is now easy to see that the event  
\[
\m B = \bigcap_{k\in \m J: k\geq \bar j} 
\l\{ \l\|  \wh\Sigma_{m,k} - \Sigma_0 \r\|\leq 3\sigma_k\sqrt{\frac{\beta}{m}}  \r\} 
\]
of probability $\geq 1 - 5d \log_2\l(\frac{2\sigma_{\mx}}{\sigma_{\mn}}\r) e^{-\beta}$ is contained in 
$\m E=\l\{  j_\ast\leq \bar j \r\}$. 
Hence, on $\m B$  
\begin{align*}
\l\| \wh\Sigma_\ast - \Sigma_0 \r\|&
\leq \| \wh\Sigma_\ast - \wh\Sigma_{m,\bar j} \| + \| \wh\Sigma_{m,\bar j} - \Sigma_0 \| \leq 
6 \sigma_{\bar j}\sqrt{\frac{\beta}{m}} + 3\sigma_{\bar j}\sqrt{\frac{\beta}{m}} \\
&\leq 12\sigma_0\sqrt{\frac{\beta}{m}} + 6\sigma_0\sqrt{\frac{\beta}{m}} = 18\sigma_0 \sqrt{\frac{\beta}{m}},
\end{align*}
and the claim follows. 

\subsection{Proof of Theorem \ref{th:covariance}}

The proof is based on the following lemma:
\begin{lemma}
Inequality (\ref{eq:ex70}) holds on the event $\m E=\l\{ \tau\geq 2\l\| \wh \Sigma_{\ast} - \Sigma_0 \r\| \r\}$. 
\end{lemma}
To verify this statement, it is enough to repeat the steps of the proof of Theorem 1 in \cite{lounici2014high}, replacing each occurrence of the sample covariance matrix by its ``robust analogue'' $\wh \Sigma_\ast$. \\
It then follows from Theorem \ref{th:lepski} that $\Pr(\m E)\geq 1 - 5d \log_2\l(\frac{2\sigma_{\mx}}{\sigma_{\mn}}\r) e^{-\beta}$ whenever $\tau \geq 36 \sigma_0 \sqrt{\frac{\beta}{m}}$.

\bibliographystyle{imsart-number}
\bibliography{bibliography}

\begin{thebibliography}{34}

\bibitem{allard2012multi}
\begin{barticle}[author]
\bauthor{\bsnm{Allard},~\bfnm{W.~K.}\binits{W.~K.}},
  \bauthor{\bsnm{Chen},~\bfnm{G.}\binits{G.}} \AND
  \bauthor{\bsnm{Maggioni},~\bfnm{M.}\binits{M.}}
(\byear{2012}).
\btitle{Multi-scale geometric methods for data sets {II}: Geometric
  multi-resolution analysis}.
\bjournal{Applied and Computational Harmonic Analysis}
\bvolume{32}
\bpages{435--462}.
\end{barticle}
\endbibitem

\bibitem{alter2000singular}
\begin{barticle}[author]
\bauthor{\bsnm{Alter},~\bfnm{Orly}\binits{O.}},
  \bauthor{\bsnm{Brown},~\bfnm{Patrick~O}\binits{P.~O.}} \AND
  \bauthor{\bsnm{Botstein},~\bfnm{David}\binits{D.}}
(\byear{2000}).
\btitle{Singular value decomposition for genome-wide expression data processing
  and modeling}.
\bjournal{Proceedings of the National Academy of Sciences}
\bvolume{97}
\bpages{10101--10106}.
\end{barticle}
\endbibitem

\bibitem{balasubramanian2016discussion}
\begin{barticle}[author]
\bauthor{\bsnm{Balasubramanian},~\bfnm{Krishnakumar}\binits{K.}} \AND
  \bauthor{\bsnm{Yuan},~\bfnm{Ming}\binits{M.}}
(\byear{2016}).
\btitle{Discussion of {``Estimating structured high-dimensional covariance and
  precision matrices: optimal rates and adaptive estimation''}}.
\bjournal{Electronic Journal of Statistics}
\bvolume{10}
\bpages{71--73}.
\end{barticle}
\endbibitem

\bibitem{bhatia2013matrix}
\begin{bbook}[author]
\bauthor{\bsnm{Bhatia},~\bfnm{Rajendra}\binits{R.}}
(\byear{2013}).
\btitle{Matrix analysis}
\bvolume{169}.
\bpublisher{Springer Science \& Business Media}.
\end{bbook}
\endbibitem

\bibitem{boucheron2013concentration}
\begin{bbook}[author]
\bauthor{\bsnm{Boucheron},~\bfnm{St{\'e}phane}\binits{S.}},
  \bauthor{\bsnm{Lugosi},~\bfnm{G{\'a}bor}\binits{G.}} \AND
  \bauthor{\bsnm{Massart},~\bfnm{Pascal}\binits{P.}}
(\byear{2013}).
\btitle{Concentration inequalities: A nonasymptotic theory of independence}.
\bpublisher{Oxford university press}.
\end{bbook}
\endbibitem

\bibitem{cai2016}
\begin{barticle}[author]
\bauthor{\bsnm{Cai},~\bfnm{T.~T.}\binits{T.~T.}},
  \bauthor{\bsnm{Ren},~\bfnm{Z.}\binits{Z.}} \AND
  \bauthor{\bsnm{Zhou},~\bfnm{H.~H.}\binits{H.~H.}}
(\byear{2016}).
\btitle{Estimating structured high-dimensional covariance and precision
  matrices: optimal rates and adaptive estimation}.
\bjournal{Electron. J. Statist.}
\bvolume{10}
\bpages{1--59}.
\bdoi{10.1214/15-EJS1081}
\end{barticle}
\endbibitem

\bibitem{catoni2012challenging}
\begin{binproceedings}[author]
\bauthor{\bsnm{Catoni},~\bfnm{O.}\binits{O.}}
(\byear{2012}).
\btitle{Challenging the empirical mean and empirical variance: a deviation
  study}.
In \bbooktitle{Annales de l'Institut Henri Poincar{\'e}, Probabilit{\'e}s et
  Statistiques}
\bvolume{48}
\bpages{1148--1185}.
\end{binproceedings}
\endbibitem

\bibitem{catoni2016pac}
\begin{barticle}[author]
\bauthor{\bsnm{Catoni},~\bfnm{O.}\binits{O.}}
(\byear{2016}).
\btitle{{PAC}-{B}ayesian bounds for the {G}ram matrix and least squares
  regression with a random design}.
\bjournal{arXiv preprint arXiv:1603.05229}.
\end{barticle}
\endbibitem

\bibitem{chen2015robust}
\begin{barticle}[author]
\bauthor{\bsnm{Chen},~\bfnm{Mengjie}\binits{M.}},
  \bauthor{\bsnm{Gao},~\bfnm{Chao}\binits{C.}} \AND
  \bauthor{\bsnm{Ren},~\bfnm{Zhao}\binits{Z.}}
(\byear{2015}).
\btitle{Robust Covariance Matrix Estimation via Matrix Depth}.
\bjournal{arXiv preprint arXiv:1506.00691}.
\end{barticle}
\endbibitem

\bibitem{davis1970rotation}
\begin{barticle}[author]
\bauthor{\bsnm{Davis},~\bfnm{Chandler}\binits{C.}} \AND
  \bauthor{\bsnm{Kahan},~\bfnm{William~Morton}\binits{W.~M.}}
(\byear{1970}).
\btitle{The rotation of eigenvectors by a perturbation. III}.
\bjournal{SIAM Journal on Numerical Analysis}
\bvolume{7}
\bpages{1--46}.
\end{barticle}
\endbibitem

\bibitem{diakonikolas2016robust}
\begin{binproceedings}[author]
\bauthor{\bsnm{Diakonikolas},~\bfnm{Ilias}\binits{I.}},
  \bauthor{\bsnm{Kamath},~\bfnm{Gautam}\binits{G.}},
  \bauthor{\bsnm{Kane},~\bfnm{Daniel~M}\binits{D.~M.}},
  \bauthor{\bsnm{Li},~\bfnm{Jerry}\binits{J.}},
  \bauthor{\bsnm{Moitra},~\bfnm{Ankur}\binits{A.}} \AND
  \bauthor{\bsnm{Stewart},~\bfnm{Alistair}\binits{A.}}
(\byear{2016}).
\btitle{Robust estimators in high dimensions without the computational
  intractability}.
In \bbooktitle{Foundations of Computer Science (FOCS), 2016 IEEE 57th Annual
  Symposium on}
\bpages{655--664}.
\bpublisher{IEEE}.
\end{binproceedings}
\endbibitem

\bibitem{fan2017robust}
\begin{barticle}[author]
\bauthor{\bsnm{Fan},~\bfnm{Jianqing}\binits{J.}} \AND
  \bauthor{\bsnm{Kim},~\bfnm{Donggyu}\binits{D.}}
(\byear{2017}).
\btitle{Robust High-dimensional Volatility Matrix Estimation for High-Frequency
  Factor Model}.
\bjournal{Journal of the American Statistical Association}.
\end{barticle}
\endbibitem

\bibitem{fan2017estimation}
\begin{barticle}[author]
\bauthor{\bsnm{Fan},~\bfnm{Jianqing}\binits{J.}},
  \bauthor{\bsnm{Li},~\bfnm{Quefeng}\binits{Q.}} \AND
  \bauthor{\bsnm{Wang},~\bfnm{Yuyan}\binits{Y.}}
(\byear{2017}).
\btitle{Estimation of high dimensional mean regression in the absence of
  symmetry and light tail assumptions}.
\bjournal{Journal of the Royal Statistical Society: Series B (Statistical
  Methodology)}
\bvolume{79}
\bpages{247--265}.
\end{barticle}
\endbibitem

\bibitem{fan2016overview}
\begin{barticle}[author]
\bauthor{\bsnm{Fan},~\bfnm{Jianqing}\binits{J.}},
  \bauthor{\bsnm{Liao},~\bfnm{Yuan}\binits{Y.}} \AND
  \bauthor{\bsnm{Liu},~\bfnm{Han}\binits{H.}}
(\byear{2016}).
\btitle{An overview of the estimation of large covariance and precision
  matrices}.
\bjournal{The Econometrics Journal}
\bvolume{19}
\bpages{C1--C32}.
\end{barticle}
\endbibitem

\bibitem{fan2016eigenvector}
\begin{barticle}[author]
\bauthor{\bsnm{Fan},~\bfnm{Jianqing}\binits{J.}},
  \bauthor{\bsnm{Wang},~\bfnm{Weichen}\binits{W.}} \AND
  \bauthor{\bsnm{Zhong},~\bfnm{Yiqiao}\binits{Y.}}
(\byear{2016}).
\btitle{An $\ell_\infty$ Eigenvector Perturbation Bound and Its Application to
  Robust Covariance Estimation}.
\bjournal{arXiv preprint arXiv:1603.03516}.
\end{barticle}
\endbibitem

\bibitem{fan2016robust}
\begin{barticle}[author]
\bauthor{\bsnm{Fan},~\bfnm{Jianqing}\binits{J.}},
  \bauthor{\bsnm{Wang},~\bfnm{Weichen}\binits{W.}} \AND
  \bauthor{\bsnm{Zhu},~\bfnm{Ziwei}\binits{Z.}}
(\byear{2016}).
\btitle{Robust Low-Rank Matrix Recovery}.
\bjournal{arXiv preprint arXiv:1603.08315}.
\end{barticle}
\endbibitem

\bibitem{fang1990symmetric}
\begin{bbook}[author]
\bauthor{\bsnm{Fang},~\bfnm{Kai-Tai}\binits{K.-T.}},
  \bauthor{\bsnm{Kotz},~\bfnm{Samuel}\binits{S.}} \AND
  \bauthor{\bsnm{Ng},~\bfnm{Kai~Wang}\binits{K.~W.}}
(\byear{1990}).
\btitle{Symmetric multivariate and related distributions}.
\bpublisher{Chapman and Hall}.
\end{bbook}
\endbibitem

\bibitem{giulini2015pac}
\begin{barticle}[author]
\bauthor{\bsnm{Giulini},~\bfnm{I.}\binits{I.}}
(\byear{2015}).
\btitle{{PAC-Bayesian} bounds for {Principal Component Analysis} in {Hilbert}
  spaces}.
\bjournal{arXiv preprint arXiv:1511.06263}.
\end{barticle}
\endbibitem

\bibitem{han2016eca}
\begin{barticle}[author]
\bauthor{\bsnm{Han},~\bfnm{Fang}\binits{F.}} \AND
  \bauthor{\bsnm{Liu},~\bfnm{Han}\binits{H.}}
(\byear{2016}).
\btitle{ECA: High dimensional elliptical component analysis in non-Gaussian
  distributions}.
\bjournal{Journal of the American Statistical Association}
\bvolume{just-accepted}.
\end{barticle}
\endbibitem

\bibitem{hotelling1933analysis}
\begin{barticle}[author]
\bauthor{\bsnm{Hotelling},~\bfnm{Harold}\binits{H.}}
(\byear{1933}).
\btitle{Analysis of a complex of statistical variables into principal
  components}.
\bjournal{Journal of educational psychology}
\bvolume{24}
\bpages{417}.
\end{barticle}
\endbibitem

\bibitem{hubert2008high}
\begin{barticle}[author]
\bauthor{\bsnm{Hubert},~\bfnm{M.}\binits{M.}},
  \bauthor{\bsnm{Rousseeuw},~\bfnm{P.~J.}\binits{P.~J.}} \AND
  \bauthor{\bsnm{Van~Aelst},~\bfnm{S.}\binits{S.}}
(\byear{2008}).
\btitle{High-breakdown robust multivariate methods}.
\bjournal{Statistical Science}
\bpages{92--119}.
\end{barticle}
\endbibitem

\bibitem{ledoit2004well}
\begin{barticle}[author]
\bauthor{\bsnm{Ledoit},~\bfnm{Olivier}\binits{O.}} \AND
  \bauthor{\bsnm{Wolf},~\bfnm{Michael}\binits{M.}}
(\byear{2004}).
\btitle{A well-conditioned estimator for large-dimensional covariance
  matrices}.
\bjournal{Journal of multivariate analysis}
\bvolume{88}
\bpages{365--411}.
\end{barticle}
\endbibitem

\bibitem{ledoit2012nonlinear}
\begin{barticle}[author]
\bauthor{\bsnm{Ledoit},~\bfnm{Olivier}\binits{O.}},
  \bauthor{\bsnm{Wolf},~\bfnm{Michael}\binits{M.}} \betal{et~al.}
(\byear{2012}).
\btitle{Nonlinear shrinkage estimation of large-dimensional covariance
  matrices}.
\bjournal{The Annals of Statistics}
\bvolume{40}
\bpages{1024--1060}.
\end{barticle}
\endbibitem

\bibitem{lepskii1992asymptotically}
\begin{barticle}[author]
\bauthor{\bsnm{Lepski},~\bfnm{O.}\binits{O.}}
(\byear{1992}).
\btitle{Asymptotically minimax adaptive estimation. {I}: Upper bounds.
  Optimally adaptive estimates}.
\bjournal{Theory of Probability \& Its Applications}
\bvolume{36}
\bpages{682--697}.
\end{barticle}
\endbibitem

\bibitem{lounici2014high}
\begin{barticle}[author]
\bauthor{\bsnm{Lounici},~\bfnm{K.}\binits{K.}}
(\byear{2014}).
\btitle{High-dimensional covariance matrix estimation with missing
  observations}.
\bjournal{Bernoulli}
\bvolume{20}
\bpages{1029--1058}.
\end{barticle}
\endbibitem

\bibitem{minsker2013geometric}
\begin{barticle}[author]
\bauthor{\bsnm{Minsker},~\bfnm{S.}\binits{S.}}
(\byear{2015}).
\btitle{Geometric median and robust estimation in {B}anach spaces}.
\bjournal{Bernoulli}
\bvolume{21}
\bpages{2308--2335}.
\end{barticle}
\endbibitem

\bibitem{minsker2016sub}
\begin{barticle}[author]
\bauthor{\bsnm{Minsker},~\bfnm{Stanislav}\binits{S.}}
(\byear{2016}).
\btitle{{Sub-Gaussian} estimators of the mean of a random matrix with
  heavy-tailed entries}.
\bjournal{arXiv preprint arXiv:1605.07129}.
\end{barticle}
\endbibitem

\bibitem{novembre2008genes}
\begin{barticle}[author]
\bauthor{\bsnm{Novembre},~\bfnm{John}\binits{J.}},
  \bauthor{\bsnm{Johnson},~\bfnm{Toby}\binits{T.}},
  \bauthor{\bsnm{Bryc},~\bfnm{Katarzyna}\binits{K.}},
  \bauthor{\bsnm{Kutalik},~\bfnm{Zolt{\'a}n}\binits{Z.}},
  \bauthor{\bsnm{Boyko},~\bfnm{Adam~R}\binits{A.~R.}},
  \bauthor{\bsnm{Auton},~\bfnm{Adam}\binits{A.}},
  \bauthor{\bsnm{Indap},~\bfnm{Amit}\binits{A.}},
  \bauthor{\bsnm{King},~\bfnm{Karen~S}\binits{K.~S.}},
  \bauthor{\bsnm{Bergmann},~\bfnm{Sven}\binits{S.}},
  \bauthor{\bsnm{Nelson},~\bfnm{Matthew~R}\binits{M.~R.}} \betal{et~al.}
(\byear{2008}).
\btitle{Genes mirror geography within {E}urope}.
\bjournal{Nature}
\bvolume{456}
\bpages{98--101}.
\end{barticle}
\endbibitem

\bibitem{saal2007poor}
\begin{barticle}[author]
\bauthor{\bsnm{Saal},~\bfnm{Lao~H}\binits{L.~H.}},
  \bauthor{\bsnm{Johansson},~\bfnm{Peter}\binits{P.}},
  \bauthor{\bsnm{Holm},~\bfnm{Karolina}\binits{K.}},
  \bauthor{\bsnm{Gruvberger-Saal},~\bfnm{Sofia~K}\binits{S.~K.}},
  \bauthor{\bsnm{She},~\bfnm{Qing-Bai}\binits{Q.-B.}},
  \bauthor{\bsnm{Maurer},~\bfnm{Matthew}\binits{M.}},
  \bauthor{\bsnm{Koujak},~\bfnm{Susan}\binits{S.}},
  \bauthor{\bsnm{Ferrando},~\bfnm{Adolfo~A}\binits{A.~A.}},
  \bauthor{\bsnm{Malmstr{\"o}m},~\bfnm{Per}\binits{P.}},
  \bauthor{\bsnm{Memeo},~\bfnm{Lorenzo}\binits{L.}} \betal{et~al.}
(\byear{2007}).
\btitle{Poor prognosis in carcinoma is associated with a gene expression
  signature of aberrant {PTEN} tumor suppressor pathway activity}.
\bjournal{Proceedings of the National Academy of Sciences}
\bvolume{104}
\bpages{7564--7569}.
\end{barticle}
\endbibitem

\bibitem{tropp1}
\begin{barticle}[author]
\bauthor{\bsnm{Tropp},~\bfnm{J.~A.}\binits{J.~A.}}
(\byear{2012}).
\btitle{User-friendly tail bounds for sums of random matrices}.
\bjournal{Found. Comput. Math.}
\bvolume{12}
\bpages{389--434}.
\bdoi{10.1007/s10208-011-9099-z}
\bmrnumber{2946459}
\end{barticle}
\endbibitem

\bibitem{tropp2015introduction}
\begin{barticle}[author]
\bauthor{\bsnm{Tropp},~\bfnm{J.~A.}\binits{J.~A.}}
(\byear{2015}).
\btitle{An introduction to matrix concentration inequalities}.
\bjournal{arXiv preprint arXiv:1501.01571}.
\end{barticle}
\endbibitem

\bibitem{tukey1975mathematics}
\begin{binproceedings}[author]
\bauthor{\bsnm{Tukey},~\bfnm{J.~W.}\binits{J.~W.}}
(\byear{1975}).
\btitle{Mathematics and the picturing of data}.
In \bbooktitle{Proceedings of the international congress of mathematicians}
\bvolume{2}
\bpages{523--531}.
\end{binproceedings}
\endbibitem

\bibitem{tyler1987distribution}
\begin{barticle}[author]
\bauthor{\bsnm{Tyler},~\bfnm{D.~E.}\binits{D.~E.}}
(\byear{1987}).
\btitle{A distribution-free {M}-estimator of multivariate scatter}.
\bjournal{The Annals of Statistics}
\bpages{234--251}.
\end{barticle}
\endbibitem

\bibitem{wegkamp2016adaptive}
\begin{barticle}[author]
\bauthor{\bsnm{Wegkamp},~\bfnm{Marten}\binits{M.}},
  \bauthor{\bsnm{Zhao},~\bfnm{Yue}\binits{Y.}} \betal{et~al.}
(\byear{2016}).
\btitle{Adaptive estimation of the copula correlation matrix for semiparametric
  elliptical copulas}.
\bjournal{Bernoulli}
\bvolume{22}
\bpages{1184--1226}.
\end{barticle}
\endbibitem

\end{thebibliography}

\normalsize
\newpage
\section{Supplement}
\subsection{Preliminaries}
\begin{lemma}\label{log-bounded-function}
Consider any function $\phi:\mathbb{R}\rightarrow\mathbb{R}$ and $\theta>0$. Suppose the following holds
\begin{equation}\label{assumption-1}
-\frac1\theta\log\left(1-\theta x+\theta^2x^2\right)\leq \phi(x)
\leq \frac1\theta\log\left(1+\theta x+\theta^2x^2\right), ~\forall x\in\mathbb{R}
\end{equation}
then, we have for any matrix $A\in\mathbb{H}^{d\times d}$,
\[
-\frac1\theta\log\left(1-\theta A+\theta^2A^2\right)\leq \phi(A)
\leq \frac1\theta\log\left(I+\theta A+\theta^2A^2\right).
\]
\end{lemma}
\begin{proof}
Note that for any $x\in\mathbb{R}$, 
$-\frac1\theta\log\left(1 - x\theta + x^2\theta^2\right)\leq\frac1\theta\log\left(1 + x\theta+ x^2\theta^2\right)$,
then, the claim follows immediately from the definition of the matrix function.
\end{proof}

The above lemma is useful in our context mainly due to the following lemma,
\begin{lemma}\label{truncation-function}
The truncation function $\frac1\theta\psi(\theta x) = \textrm{sign}(x)\cdot\l(|x|\wedge\frac1\theta\r)$ satisfies the assumption \eqref{assumption-1} in Lemma \ref{log-bounded-function}.
\end{lemma}
\begin{proof}
Denote $f_1(x) = -\frac1\theta\log\left(1-\theta x+\theta^2x^2\right)$, $f_2(x) =  \frac1\theta\log\left(1+\theta x+\theta^2x^2\right)$ and $g(x) = \textrm{sign}(x)\cdot\l(|x|\wedge\frac1\theta\r)$. Note first that 
\begin{align*}
&f_1(0) = g(0) = f_2(0) = 0,\\
&f_1(1/\theta)\leq g(1/\theta) \leq f_2(1/\theta),\\
&f_1(-1/\theta)\leq g(-1/\theta) \leq f_2(-1/\theta),\\
\end{align*}
and the subgradient
\[
\partial g(x) = 
\begin{cases}
1, &~~x\in(-1/\theta,1/\theta),\\
0, &~~x\in(-\infty,-1/\theta)\cup(1/\theta,+\infty),\\
[0,1], &~~ x = -1/\theta,1/\theta.
\end{cases}
\]
Next, we take the derivative of $f_2(x)$ and compare it to the derivative of $g(x)$.
\[
f_2'(x) = \frac1\theta\cdot\frac{\theta+2x\theta^2}{1+x\theta + x^2\theta^2} 
=\frac{1+2x\theta}{1+ x\theta+x^2\theta^2}.
\]
Note that $f_2'(x)\geq1,x\in(0,1/\theta)$, $f_2'(x)\geq0,x\geq1/\theta$, $f_2'(x)\leq1,x\in(-1/\theta,0]$ and 
$f_2'(x)\leq0,x\leq-1/\theta$. Thus, we have $g(x)\leq f_2(x),~\forall x\in\mathbb{R}$. Similarly, we can take the derivative of $f_1(x)$ and compare it to $g(x)$, which results in $f_1'(x)\leq1,x\in(0,1/\theta)$, $f_1'(x)\leq0,x\geq1/\theta$, $f_1'(x)\geq1,x\in(-1/\theta,0]$ and 
$f_2'(x)\geq0,x\leq-1/\theta$. This implies $f_1(x)\leq g(x)$ and the Lemma is proved.
\end{proof}

The following lemma demonstrates the importance of matrix logarithm function in matrix analysis, whose proof can be found in \cite{bhatia2013matrix} and \cite{tropp2015introduction},
\begin{lemma}
(a) The matrix logarithm is operator monotone, that is, 
if $A\succ B\succ0$ are two matrices in $\mathbb{H}^{d\times d}$, then, $\log(A)\succ\log(B)$.\\
(b) Given a fixed matrix $H\in \mathbb{H}^{d\times d}$, the function
\[
A\rightarrow tr\exp(H+\log(A))
\]
is concave on the cone of positive semi-definite matrices.
\end{lemma}

The following lemma is a generalization of Chebyshev's association inequality. See Theorem 2.15 of \cite{boucheron2013concentration} for proof.
\begin{lemma}[FKG inequality]
Suppose $f,g:\mathbb{R}^d\rightarrow\mathbb{R}$ are two functions non-decreasing on each coordinate. Let $Y=[Y_1,~Y_2,~\cdots,~Y_d]$ be a random vector taking values in $\mathbb{R}^d$, then, 
\[
\expect{f(X)g(X)}\geq\expect{f(X)}\expect{g(X)}.
\]
\end{lemma}
The following corollary follows immediately from the FKG inequality.
\begin{corollary}\label{FKG-bound}
Let $Z=X-\mu_0$, then, we have 
$\sigma_0^2 = \|\expect{ZZ^T\|Z\|_2^2}\|\geq tr\left(\expect{ZZ^T}\right)\left\|\expect{ZZ^T}\right\|
=tr(\Sigma_0)\|\Sigma_0\|$.
\end{corollary}
\begin{proof}
Consider any unit vector $\mathbf{v}\in\mathbb{R}^d$. It is enough to show $\expect{(\mathbf{v}^TZ)^2\|Z\|_2^2}\geq\expect{(\mathbf{v}^TZ)^2}\expect{\|Z\|_2^2}$. We change the coordinate by considering an orthonormal basis $\{\mathbf{v}_1,\cdots,\mathbf{v}_d\}$ with 
$\mathbf{v}_1=\mathbf{v}$. Let $Y_i = \mathbf{v}_i^TZ$, $i=1,2,\cdots,d$, then we obtain,
\[
\expect{(\mathbf{v}^TZ)^2\|Z\|_2^2} = \expect{Y_1^2\|Y\|_2^2}\geq\expect{Y_1^2}\expect{\|Y\|_2^2},
\]
where the last inequality follows from FKG inequality by taking $f\left(Y_1^2,~\cdots,~Y_d^2\right) = Y_1^2$ and $g\left(Y_1^2,~\cdots,~Y_d^2\right) = \|Y\|_2^2$.
\end{proof}

\subsection{Additional computation in the proof of Lemma \ref{lemma:main}}
In order to show \eqref{need-steps}, it is enough to show that
\begin{align*}
C'\|\Sigma_0\|  \l( \sqrt{\frac{\overline{d}\sigma}{\|\Sigma_0\|}}\l(\frac{\beta}{m}\r)^{\frac34} + \frac{\sqrt{\overline{d}}\sigma}{\|\Sigma_0\|}\frac{\beta}{m}
+ \sqrt{\frac{\overline{d}\sigma}{\|\Sigma_0\|}}\l(\frac{\beta}{m}\r)^{\frac54}  
+\overline{d}
\l(\frac{\beta}{m}\r)^{\frac32} + \frac{\overline{d}\beta^2}{m^2} 
+ \overline{d}^{\frac54}\l(\frac{\beta}{m}\r)^{\frac94} \right)
\leq\sigma\sqrt{\frac\beta m}.
\end{align*}
Note that $\overline{d} = \sigma_0^2/\|\Sigma_0\|^2\geq\tr(\Sigma_0)/\|\Sigma_0\|\geq1$, and
assuming that the sample size satisfies $m\geq(6C')^4\overline{d}\beta$, we have 
$\overline{d}\beta/m\leq1/(6C')^4<1$.
We then bound each of the 6 terms on the left side. 
\begin{align*}
C'\|\Sigma_0\|\sqrt{\frac{\overline{d}\sigma}{\|\Sigma_0\|}}\l(\frac{\beta}{m}\r)^{\frac34}
=&C'\sqrt{\sigma}\l(\frac\beta m\r)^{\frac14}\cdot\l(\frac{\|\Sigma_0\|\overline{d}\beta}{m}\r)^{1/4}
\cdot\l(\frac{\|\Sigma_0\|\overline{d}\beta}{m}\r)^{1/4}\\
\leq&C'\sqrt{\sigma}\l(\frac\beta m\r)^{\frac14}\cdot\l(\frac{\|\Sigma_0\|\overline{d}\beta}{m}\r)^{1/4}
\cdot\frac{1}{6C'}\\
=&\frac16\sqrt{\sigma\sigma_0}\sqrt{\frac{\beta}{m}} \leq \frac16\sigma\sqrt{\frac{\beta}{m}},\\
C'\|\Sigma_0\| \cdot\sqrt{\overline{d}}\frac{\sigma}{\|\Sigma_0\|}\frac{\beta}{m}~~~~
=&C'\sigma\sqrt{\frac\beta m}\cdot\sqrt{\frac{\overline{d}\beta}{m}}
\leq C'\sigma\sqrt{\frac\beta m}\frac{1}{(6C')^2}\leq\frac16\sigma\sqrt{\frac{\beta}{m}},\\
C'\|\Sigma_0\|\sqrt{\frac{\overline{d}\sigma}{\|\Sigma_0\|}}\l(\frac{\beta}{m}\r)^{\frac54}
\leq&C'\|\Sigma_0\|\sqrt{\frac{\overline{d}\sigma}{\|\Sigma_0\|}}\l(\frac{\beta}{m}\r)^{\frac34}
\leq\frac16\sigma\sqrt{\frac{\beta}{m}}.
\end{align*}
Note that we have the following
\begin{multline*}
C'\|\Sigma_0\|\overline{d}\frac\beta m 
= C'\|\Sigma_0\|\l(\frac{\overline{d}\beta}{m} \r)^{\frac12}\l(\frac{\overline{d}\beta}{m} \r)^{\frac12}
\leq C'\|\Sigma_0\|\l(\frac{\overline{d}\beta}{m} \r)^{\frac12}\frac{1}{(6C')^2}
\leq\frac16\sigma_0\sqrt{\frac\beta m}\leq\frac16\sigma\sqrt{\frac\beta m},
\end{multline*}
thus, the rest three terms can be bounded as follows,
\begin{align*}
C'\|\Sigma_0\|\overline{d}\l(\frac{\beta}{m}\r)^{\frac32}
\leq& C'\|\Sigma_0\|\overline{d}
\frac{\beta}{m}\leq \frac16\sigma\sqrt{\frac\beta m}\\
C'\|\Sigma_0\|\overline{d}\frac{\beta^2}{m^2} ~~~~
\leq& C'\|\Sigma_0\|\overline{d}
\frac{\beta}{m}\leq \frac16\sigma\sqrt{\frac\beta m}\\
C'\|\Sigma_0\|\overline{d}^{\frac54}\l(\frac{\beta}{m}\r)^{\frac94}
\leq&C'\|\Sigma_0\|\overline{d}^{\frac54}\l(\frac{\beta}{m}\r)^{\frac54}
\leq C'\|\Sigma_0\|\overline{d}
\frac{\beta}{m}\leq \frac16\sigma\sqrt{\frac\beta m}.
\end{align*}
Overall, we have \eqref{need-steps} holds.

\subsection{Proof of Lemma \ref{key-lemma}}
First of all, by definition of $\wh\Sigma_\mu$, we have
\begin{align*}
\sup_{\|\mu\|_2\leq B_\beta, \|\mathbf{v}\|_2\leq 1}\left| \mathbf{v}^T(\hat{\Sigma}_\mu - \Sigma_\mu)\mathbf{v} \right|
= \sup_{\|\mu\|_2\leq B_\beta, \|\mathbf{v}\|_2\leq 1}\left| \frac{1}{m\theta}\sum_{i=1}^m\dotp{Z_i-\mu}{\mathbf{v}}^2
\frac{\psi\l(\theta\|Z_i-\mu\|_2^2\r)}{\|Z_i-\mu\|_2^2} - 
\expect{\dotp{Z_i-\mu}{\mathbf{v}}^2} \right|.
\end{align*}
Expanding the squares on the right hand side gives
\begin{align*}
\sup_{\|\mu\|_2\leq B_\beta}\left\| \hat{\Sigma}_\mu - \Sigma_\mu \right\|
\leq& 
\sup_{\|\mu\|_2\leq B_\beta, \|\mathbf{v}\|_2\leq 1}
\left| \frac1m\sum_{i=1}^m\dotp{Z_i}{\mathbf{v}}^2
\frac{\psi\l(\theta\|Z_i-\mu\|_2^2\r)}{\theta\|Z_i-\mu\|_2^2} - \expect{\dotp{Z_i}{\mathbf{v}}^2}\right| ~~\text{(I)}\\
&+ 2 \sup_{\|\mu\|_2\leq B_\beta, \|\mathbf{v}\|_2\leq 1}
\left| \frac1m\sum_{i=1}^m\dotp{Z_i}{\mathbf{v}}\dotp{\mu}{\mathbf{v}}
\frac{\psi\l(\theta\|Z_i-\mu\|_2^2\r)}{\theta\|Z_i-\mu\|_2^2} - \expect{\dotp{Z_i}{\mathbf{v}}\dotp{\mu}{\mathbf{v}}}\right| ~~\text{(II)}\\
& + \sup_{\|\mu\|_2\leq B_\beta, \|\mathbf{v}\|_2\leq 1}
\left| \frac1m\sum_{i=1}^m\dotp{\mu}{\mathbf{v}}^2
\frac{\psi\l(\theta\|Z_i-\mu\|_2^2\r)}{\theta\|Z_i-\mu\|_2^2} - \dotp{\mu}{\mathbf{v}}^2\right|.~~\text{(III)}
\end{align*}
We will then bound these three terms separately. Note that given $\|\wh\mu-\mu_0\|_2\leq B_\beta$, the term (III) can be readily bounded as follows using the fact that $0\leq\psi(x)\leq x,~\forall x\geq0$,
\begin{multline}\label{final-bound-III}
\text{(III)} = \sup_{\|\mu\|_2\leq B_\beta, \|\mathbf{v}\|_2\leq 1}
\left| \dotp{\mu}{\mathbf{v}}^2\left(\frac1m\sum_{i=1}^m\frac{\psi\l(\theta\|Z_i-\mu\|_2^2\r)}{\theta\|Z_i-\mu\|_2^2}  -  1\right) \right|
\leq \sup_{\|\mu\|_2\leq B_\beta, \|\mathbf{v}\|_2\leq 1}\dotp{\mu}{\mathbf{v}}^2
\leq B_\beta^2 \\
= 242\frac{tr(\Sigma_0)}{m}\beta \leq 242\frac{\sigma_0^2\beta}{\|\Sigma_0\|m}
\leq 242\|\Sigma_0\|\frac{\overline{d}\beta}{m},
\end{multline}
where the second from the last inequality follows from Corollary \ref{FKG-bound} and the last inequality follows from $\overline{d}=\sigma_0^2/\|\Sigma_0\|^2$.

The rest two terms are bounded through the following lemma whose proof is delayed to the next section:
\begin{lemma}\label{key-sublemma}
Given $\|\wh\mu-\mu_0\|_2\leq B_\beta$, with probability at least $1-4de^{-\beta}$, we have the following two bounds hold,
\begin{align*}
\text{(I)}\leq 2\sigma\sqrt{\frac{\beta}{m}} 
+22\|\Sigma_0\|  \l( \sqrt{2}\overline{d}^{\frac14}\l(\frac{\beta}{m}\r)^{\frac34}
+ 2\sqrt{2}\sqrt{\frac{\overline{d}\sigma}{\|\Sigma_0\|}}\l(\frac{\beta}{m}\r)^{\frac54}  
+11\overline{d}^{\frac12}
\l(\frac{\beta}{m}\r)^{\frac32} + 22\frac{\overline{d}\beta^2}{m^2}\r),
\end{align*}
\begin{multline*}
\text{(II)} \leq 11\|\Sigma_0\|  \l( \sqrt2\sqrt{\frac{\overline{d}\sigma}{\|\Sigma_0\|}}\l(\frac{\beta}{m}\r)^{\frac34} + 3\sqrt2\sqrt{\overline{d}}\frac{\sigma}{\|\Sigma_0\|}\frac{\beta}{m} 
+ 44\overline{d}^{\frac34}\l(\frac{\beta}{m}\r)^{\frac54}  \r.\\
\l.
+44\sqrt2\overline{d}
\l(\frac{\beta}{m}\r)^{\frac32} + 242\sqrt2\frac{\overline{d}\beta^2}{m^2} + 484\overline{d}^{\frac54}\l(\frac{\beta}{m}\r)^{\frac94} \r).
\end{multline*}
\end{lemma}
Note that since $\sigma\geq\sigma_0$, we have
$\sigma/\|\Sigma_0\|\geq\sigma_0/\|\Sigma_0\|=\sqrt{\overline{d}}$.
Combining the above lemma with \eqref{final-bound-III}
finishes the proof of Lemma \ref{key-lemma}.

\subsection{Proof of Lemma \ref{key-sublemma}}
Before proving the Lemma, we introduce the following abbreviations:
\begin{align*}
&g_\mathbf{v}(Z_i) = \dotp{Z_i}{\mathbf{v}}^2\frac{\psi\l(\theta\|Z_i\|_2^2\r)}{\theta\|Z_i\|_2^2},
~~h_\mu(Z_i) = \frac{\|Z_i\|_2^2}{\psi\l(\theta\|Z_i\|_2^2\r)}\frac{\psi\l(\theta\|Z_i-\mu\|_2^2\r)}{\|Z_i-\mu\|_2^2},\\
&\tilde{g}_\mathbf{v}(Z_i) = \dotp{Z_i}{\mathbf{v}}\frac{\psi\l(\theta\|Z_i\|_2^2\r)}{\theta\|Z_i\|_2^2}.
\end{align*}
Our analysis relies on the following simply yet important fact which gives deterministic upper and lower bound of $h_\mu(Z_i)$ around 1. Its proof is delayed to the next section.

\begin{lemma}\label{ratio-bound}
For any $\mu$ such that $\|\mu\|_2\leq B_\beta$, the following holds:
\[
1 - 2B_\beta\sqrt{\theta} - B_\beta^2\theta\leq h_\mu(Z_i) \leq 1+ 2B_\beta\sqrt{\theta} + B_\beta^2\theta.
\]
\end{lemma}

The following Lemma gives a general concentration bound for heavy tailed random matrices under a mapping $\phi(\cdot)$. 
\begin{lemma}\label{concentration-bound}
Let $A_1,~A_2,\cdots,~A_m$ be a sequence of i.i.d. random matrices in $\mathbb{H}^{d\times d}$ with zero mean and finite second moment $\sigma_A = \|\expect{A_i^2}\|$. Let $\phi(\cdot)$ be any function satisfying the assumption \eqref{assumption-1} of Lemma \ref{log-bounded-function}. Then, for any $t>0$,
\[
Pr\left(\sum_{i=1}^m\left(\phi(A_i) - \expect{A_i}\right)\geq t \sqrt{m}\right)
\leq 2d\exp\left( -t\theta\sqrt{m} + m\theta^2\sigma_A^2 \right).
\]
Specifically, if the assumption \eqref{assumption-1} holds for $\theta = \frac{t}{2\sqrt{m}\sigma_A^2}$, then we obtain the subgaussian tail 
$2d\exp(-t^2/4\sigma_A^2)$.
\end{lemma}
The intuition behind this lemma is that the $\log(1+x)$ tends to ``robustify'' a random variable by implicitly trading the bias for a tight concentration. A scalar version of such lemma with a similar idea is first introduced in the seminal work \cite{catoni2012challenging}. 
The proof of the current matrix version is similar to Lemma 3.1 and Theorem 3.1 of \cite{minsker2016sub} by modifying only the constants. We omitted the details here for brevity. 
Note that this lemma is useful in our context by choosing $\phi(x) = \frac1\theta\psi(\theta x)$. Next, we prove two parts of Lemma \ref{key-sublemma} separately.

\begin{proof}[Proof of (I) in Lemma \ref{key-sublemma}]
 Using the abbreviation introduced at the beginning of this section, we have
\[
(I) = \sup_{\|\mu\|_2\leq B_\beta, \|\mathbf{v}\|_2\leq 1}
\left| \frac1m\sum_{i=1}^m g_\mathbf{v}(Z_i)h_\mu(Z_i) - \expect{\dotp{Z_i}{\mathbf{v}}^2} \right|
\]

We further split it into two terms as follows:
\begin{equation}\label{bound-of-I}
(I) \leq \sup_{\|\mu\|_2\leq B_\beta, \|\mathbf{v}\|_2\leq 1}
\left| \frac1m\sum_{i=1}^m g_\mathbf{v}(Z_i)\left(h_\mu(Z_i) - 1\right) \right|
+ \sup_{\|\mathbf{v}\|\leq}\left| \frac1m\sum_{i=1}^m g_\mathbf{v}(Z_i) -  \expect{\dotp{Z_i}{\mathbf{v}}^2} \right|
\end{equation}
The two terms in \eqref{bound-of-I} are bounded as follows:
\begin{enumerate}
\item For the second term in \eqref{bound-of-I}, note that we can write it back into the matrix form as 
\[
\left\| \frac{1}{m\theta}\sum_{i=1}^m Z_iZ_i^T\frac{\psi\l(\theta\|Z_i\|_2^2\r)}{\|Z_i\|_2^2}-\expect{Z_iZ_i^T}\right\|.
\]
Note that the matrix $Z_iZ_i^T$ is a rank one matrix with the eigenvalue equal to $\|Z_i\|_2^2$, so it follows from the definition of matrix function,
\[
Z_iZ_i^T\frac{\psi\l(\theta\|Z_i\|_2^2\r)}{\|Z_i\|_2^2}=\frac1\theta\psi\l(\theta Z_iZ_i^T\r).
\]
Now, applying Lemma \ref{truncation-function} setting $\theta = \frac{t}{2\sigma^2\sqrt{m}}$ together with Lemma \ref{concentration-bound} gives
\begin{equation*}
Pr\left(\left\| \frac{1}{m\theta}\sum_{i=1}^m Z_iZ_i^T\frac{\psi\l(\theta\|Z_i\|_2^2\r)}{\|Z_i\|_2^2}-\expect{Z_iZ_i^T}\right\|\geq t/\sqrt{m} \right)\leq 2d\exp(-t^2/4\sigma^2).
\end{equation*}
Setting $t = 2\sigma\sqrt{\beta}$ (which results in $\theta = \frac1\sigma\sqrt{\frac{\beta}{m}}$) gives
\begin{equation}\label{inter-bound-I}
\left\| \frac{1}{m\theta}\sum_{i=1}^m Z_iZ_i^T\frac{\psi\l(\theta\|Z_i\|_2^2\r)}{\|Z_i\|_2^2}-\expect{Z_iZ_i^T}\right\| 
\leq 2\sigma\sqrt{\frac{\beta}{m}}
\end{equation}
with probability at least $1-2de^{-\beta}$.

\item For the first term in \eqref{bound-of-I}, by the fact that $g_{\mathbf{v}}(Z_i)\geq0$ and Lemma \ref{ratio-bound}, 
\begin{align*}
&\sup_{\|\mu\|_2\leq B_\beta, \|\mathbf{v}\|_2\leq 1}
\left| \frac1m\sum_{i=1}^m g_\mathbf{v}(Z_i)\left(h_\mu(Z_i)-1\right) \right|\\
&\leq \sup_{\|\mu\|_2\leq B_\beta, \|\mathbf{v}\|_2\leq 1}
 \frac1m\sum_{i=1}^m g_\mathbf{v}(Z_i) \l|h_\mu(Z_i) - 1\r|\\
 &\leq \sup_{\|\mathbf{v}\|_2\leq 1}
 \frac1m\sum_{i=1}^m g_\mathbf{v}(Z_i)\l( 2B_\beta\sqrt{\theta} + B_\beta^2\theta \r)\\
 &\leq \left(\l\|\expect{Z_iZ_i^T}\r\| + 2\sigma\sqrt{\frac{\beta}{m}}\right)\l(2B_\beta\sqrt{\theta} + B_\beta^2\theta \r)  ,
\end{align*}
with probability at least $1 - 2de^{-\beta}$,
where the last inequality follows from the same argument leading to \eqref{inter-bound-I}. Note that $\expect{Z_iZ_i^T} = \Sigma_0$.
\end{enumerate}
Overall, we get
\begin{equation*}
\text{(I)} \leq 2\sigma\sqrt{\frac{\beta}{m}} + \left(\l\|\Sigma_0\r\| + 2\sigma\sqrt{\frac{\beta}{m}}\right)\l( 
2B_\beta\sqrt{\theta} + B_\beta^2\theta\r),
\end{equation*}
with probability at least $1-2de^{-\beta}$. Now we substitute $B_\beta = 11\sqrt{2\tr(\Sigma_0)\beta/m}$ and $\theta = \frac1\sigma\sqrt{\frac{\beta}{m}}$ into the above bound gives
\begin{multline*}
\text{(I)} \leq 2\sigma\sqrt{\frac{\beta}{m}} + 22\sqrt{2}\|\Sigma_0\|\sqrt{\frac{\tr(\Sigma_0)}{\sigma}}\l(\frac{\beta}{m}\r)^{\frac34} + 242\|\Sigma_0\|\frac{\tr\Sigma_0}{\sigma}\l(\frac\beta m\r)^{\frac32}\\
+44\sqrt{2}\sqrt{\sigma\tr(\Sigma_0)}\l(\frac\beta m\r)^{\frac54} + 484\tr(\Sigma_0)\l(\frac\beta m\r)^2
\end{multline*}
Using Corollary \ref{FKG-bound}, we have 
\begin{equation}\label{support-1}
\frac{\tr(\Sigma_0)}{\sigma}\leq \frac{\tr(\Sigma_0)}{\sigma_0}\leq
\frac{\tr(\Sigma_0)}{\sqrt{\tr(\Sigma_0)\|\Sigma_0\|}}\leq \frac{\sigma_0}{\|\Sigma_0\|}\leq\overline{d},
\end{equation}
and also, 
\begin{equation}\label{support-2}
\tr(\Sigma_0)\leq\|\Sigma_0\|\sigma_0^2/\|\Sigma_0\|^2\leq\|\Sigma_0\|\overline{d}.
\end{equation}
Substitute these two bounds into the bound of (I) gives the final bound for (I) stated in Lemma 
\ref{key-sublemma} with probability at least $1-2de^{-\beta}$.
\end{proof}

\begin{proof}[Proof of (II) in Lemma \ref{key-sublemma}]
First of all, using the definition of $\tilde{g}_\mathbf{v}(Z_i)$ and $h_\mu(Z_i)$, we can rewrite (II) as follows:
\begin{align*}
\text{(II)} =& \sup_{\|\mu\|_2\leq B_\beta, \|\mathbf{v}\|_2\leq 1}
\left| \frac1m\sum_{i=1}^m   \tilde{g}_\mathbf{v}(Z_i)h_\mu(Z_i)\dotp{\mu}{\mathbf{v}}
 - \expect{\dotp{Z_i}{\mathbf{v}}}\dotp{\mu}{\mathbf{v}}\right|  \\
 \leq& B_\beta  \cdot  \sup_{\|\mu\|_2\leq B_\beta, \|\mathbf{v}\|_2\leq 1}
\left| \frac1m\sum_{i=1}^m   \tilde{g}_\mathbf{v}(Z_i)h_\mu(Z_i)
 - \expect{\dotp{Z_i}{\mathbf{v}}}\right|.
\end{align*}
Similar to the analysis of (I), we further split the above term into two terms and get
\begin{align}
(II) \leq \underbrace{B_\beta  \sup_{\|\mu\|_2\leq B_\beta, \|\mathbf{v}\|_2\leq 1}
\l|  \frac1m\sum_{i=1}^m   \tilde{g}_\mathbf{v}(Z_i)\l(h_\mu(Z_i)-1\r) \r|}_{(IV)}
+ \underbrace{B_\beta  \sup_{\|\mathbf{v}\|_2\leq 1}
\l| \frac1m\sum_{i=1}^m \tilde{g}_\mathbf{v}(Z_i) - \expect{\dotp{Z_i}{\mathbf{v}}} \r|}_{(V)}.
\end{align}
For the first term, by Cauchy-Schwarz inequality and then Lemma \ref{ratio-bound}, we get
\begin{align*}
\text{(IV)}
\leq&
B_\beta  \sup_{\|\mu\|_2\leq B_\beta, \|\mathbf{v}\|_2\leq 1}
\frac1m\sum_{i=1}^m \l|\tilde{g}_\mathbf{v}(Z_i)\l(h_\mu(Z_i)-1\r) \r|\\
\leq&
B_\beta  \sup_{\|\mu\|_2\leq B_\beta, \|\mathbf{v}\|_2\leq 1}
\l(\frac1m\sum_{i=1}^m  \tilde{g}_\mathbf{v}(Z_i)^2\r)^{1/2} 
\l(\frac1m\sum_{i=1}^m \l|h_\mu(Z_i)-1 \r|^2\r)^{1/2}\\
\leq& 
B_\beta  \sup_{\|\mathbf{v}\|_2\leq 1} \l(\frac1m\sum_{i=1}^m  \tilde{g}_\mathbf{v}(Z_i)^2\r)^{1/2}
\l( 2B_\beta\sqrt{\theta} + B_\beta^2\theta \r). 
\end{align*}
Note that $\frac1\theta\psi\l(\theta\|Z_i\|_2^2\r)/\|Z_i\|_2^2\leq1$, then, it follows,
\[
\tilde{g}_\mathbf{v}(Z_i)^2=\dotp{Z_i}{\mathbf{v}}^2\l(\frac{\frac1\theta\psi\l(\theta\|Z_i\|_2^2\r)}{\|Z_i\|_2^2}\r)^2\leq\dotp{Z_i}{\mathbf{v}}^2\frac{\frac1\theta\psi\l(\theta\|Z_i\|_2^2\r)}{\|Z_i\|_2^2}. 
\]
Thus, by the same analysis leading to \eqref{inter-bound-I}, we get 
\begin{equation}\label{first-term-bound}
\text{(IV)}\leq
B_\beta\left(\l\|\expect{Z_iZ_i^T}\r\| + 2\sigma\sqrt{\frac{\beta}{m}}\right)^{1/2}\l( 2B_\beta\sqrt{\theta} + B_\beta^2\theta \r),
\end{equation}
with probability at least $1-2de^{-\beta}$. For the second term (V), notice that $\expect{Z_i} = 0$, thus we have 
\begin{multline}\label{inter-bound-V}
\text{(V)}\leq
B_\beta  \sup_{\|\mathbf{v}\|_2\leq 1}
\l| \dotp{\frac1m\sum_{i=1}^m \frac{Z_i}{\|Z_i\|_2^2}\frac1\theta\psi(\theta\|Z_i\|_2^2)}{\mathbf{v}}  \r|  
\leq B_\beta\l\|\frac1m\sum_{i=1}^m \frac{Z_i}{\|Z_i\|_2^2}\|Z_i\|_2^2\wedge\frac1\theta\r\|_2\\
\leq B_\beta\l\|\frac1m\sum_{i=1}^m \frac{Z_i}{\|Z_i\|_2^2}\|Z_i\|_2^2\wedge\frac1\theta 
- \expect{\frac{Z_i}{\|Z_i\|_2^2}\|Z_i\|_2^2\wedge\frac1\theta}\r\|_2
+ B_\beta \l\|\expect{\frac{Z_i}{\|Z_i\|_2^2}\|Z_i\|_2^2\wedge\frac1\theta}\r\|_2. 
\end{multline}

For the second term, which measures the bias, we have by the fact $\expect{Z_i} = 0$,
\begin{multline*}
\l\|\expect{\frac{Z_i}{\|Z_i\|_2^2}\|Z_i\|_2^2\wedge\frac1\theta}\r\|_2
=\l\| \expect{Z_i\l( \frac{\|Z_i\|_2^2\wedge\frac1\theta}{\| Z_i \|_2^2} - 1 \r)} \r\|_2
=\sup_{\|\mathbf{v}\|_2\leq1}\expect{\dotp{Z_i}{\mathbf{v}}\l( \frac{\|Z_i\|_2^2\wedge\frac1\theta}{\| Z_i \|_2^2} - 1 \r)}\\
\leq\sup_{\|\mathbf{v}\|_2\leq1}\expect{\dotp{Z_i}{\mathbf{v}}1_{\{\|Z_i\|_2\geq1/\sqrt{\theta}\}}}.
\end{multline*}
Now by Cauchy-Schwarz inequality and then Markov inequality, we obtain,
\begin{multline*}
\sup_{\|\mathbf{v}\|_2\leq1}\expect{\dotp{Z_i}{\mathbf{v}}1_{\{\|Z_i\|_2\geq1/\sqrt\theta\}}}
\leq\sqrt{\sup_{\|\mathbf{v}\|_2\leq1}\expect{\dotp{Z_i}{\mathbf{v}}^2}} Pr(\|Z_i\|_2\geq1/\sqrt{\theta})^{1/2}
\leq\sqrt{\|\Sigma_0\|}\expect{\|Z_i\|_2^2}^{1/2}\sqrt{\theta}\\
= \sqrt{\|\Sigma_0\|}\frac{\tr(\Sigma_0)^{1/2}\beta^{1/4}}{m^{1/4}\sigma^{1/2}}
\leq\frac{(\|\Sigma_0\|tr(\Sigma_0))^{1/4}\beta^{1/4}}{m^{1/4}}
\leq\l( \frac{\sigma^2}{m}\beta \r)^{1/4},
\end{multline*}
where the last two inequalities both follow from Lemma \ref{FKG-bound}. This gives the second term in \eqref{inter-bound-V} is given by $B_\beta\l( \frac{\sigma^2}{m}\beta \r)^{1/4}$.

For the first term in \eqref{inter-bound-V}, note that for any vector $\mathbf{x}\in\mathbb{R}^d$, 
\[
\|\mathbf{x}\|_2 = \left\| 
\l[
\begin{matrix}
0 & \mathbf{x}^T\\
\mathbf{x} & 0
\end{matrix}
\r]
 \right\|,
\]
and furthermore, the matrix 
$\l[
\begin{matrix}
0 & \mathbf{x}^T\\
\mathbf{x} & 0
\end{matrix}
\r]$
has two same eigenvalues equal to $\|\mathbf{x}\|_2$, which follows from
\[
\l[
\begin{matrix}
0 & \mathbf{x}^T\\
\mathbf{x} & 0
\end{matrix}
\r]^2
=
\l[
\begin{matrix}
\|\mathbf{x}\|_2^2 & 0\\
0 & \mathbf{x}\mathbf{x}^T
\end{matrix}
\r].
\]
Thus, if we take

\[
A_i = 
\l[
\begin{matrix}
0 & Z_i^T\\
Z_i & 0
\end{matrix}
\r]
\frac{\|Z_i\|_2^2\wedge\frac1\theta}{\|Z_i\|_2^2},
\]
Then, the first term of \eqref{inter-bound-V} is equal to
$
\l\| \frac1m\sum_{i=1}^mA_i - \expect{ A_i} \r\|
$. For this $A_i$, we have
\[
\|\expect{A_i^2}\|\leq \expect{\|Z_i\|_2^2} = tr(\Sigma_0),
~~\|A_i\|\leq\frac{1}{\sqrt{\theta}} = \frac{m^{1/4}\sigma^{1/2}}{\beta^{1/4}}.
\]
By matrix Bernstein's inequality (\cite{tropp1}), we obtain the bound
\begin{align*}
Pr\l( \l\|\frac1m\sum_{i=1}^mA_i -\expect{A_i}\r\| \geq t \r)
\leq d \exp\l( -\frac38\l( \frac{mt^2}{\sigma^2}\wedge m\sqrt{\theta}t \r) \r)
= d \exp\l( -\frac38\l( \frac{mt^2}{\sigma^2}\wedge\frac{m^{3/4}\beta^{1/4}t}{\sigma^{1/2}} \r) \r),
\end{align*}
where $c$ is a fixed positive constant.
Taking $t = 3\sqrt{\frac{\sigma^2\beta}{\|\Sigma_0\|m}}$ gives 
\begin{multline*}
Pr\l( \l\|\frac1m\sum_{i=1}^mA_i -\expect{A_i}\r\| \geq 3\sqrt{\frac{\sigma^2}{m}\beta} \r)
\leq d \exp\l(-3\beta   \wedge  \l(m^{1/4}\beta^{3/4}\overline{d}^{1/4}\r) \r)
\leq d \exp(-\beta),
\end{multline*}
where $\overline{d} = \sigma^2/\|\Sigma_0\|^2\geq\sigma_0^2/\|\Sigma_0\|^2\geq\tr(\Sigma_0)/\|\Sigma_0\|\geq1$ and
 the last inequality follows from the assumption that $m\geq\beta$.
Overall, term (V) is bounded as follows
\[\text{(V)}\leq B_\beta\l( \frac{\sigma^2}{m}\beta \r)^{1/4} + 3B_\beta \sqrt{\frac{\sigma^2\beta}{\|\Sigma_0\|m}},\]
with probability at least $1-de^{-\beta}$. Note that $\expect{Z_iZ_i^T} = \Sigma_0$, then, combining with \eqref{first-term-bound},
the term (II) is bounded as

\begin{equation*}
\text{(II)}\leq B_\beta\left(\l\|\Sigma_0\r\|^{\frac12} + \sqrt2\sigma^{\frac12}\l(\frac{\beta}{m}\r)^{\frac14}\right)\l( 2B_\beta\sqrt\theta + B_\beta^2\theta \r)   +  B_\beta\l( \frac{\sigma^2}{m}\beta \r)^{1/4} + 3B_\beta \sqrt{\frac{\sigma^2\beta}{\|\Sigma_0\|m}},
\end{equation*}
with probability at least $1-2de^{-\beta}$. Substituting $B_\beta=11\sqrt{\frac{2\tr(\Sigma_0)\beta}{m}}$ and $\theta=\frac1\sigma\sqrt{\frac{\beta}{m}}$ gives
\begin{multline*}
\text{(II)}\leq 11\sqrt2\sqrt{\tr(\Sigma_0)\sigma}\l(\frac\beta m\r)^{\frac34}
+33\sqrt2\frac{\sqrt{\tr(\Sigma_0)}\sigma}{\|\Sigma_0\|^{1/2}}\frac\beta m
+484\|\Sigma_0\|^{1/2}\frac{\tr(\Sigma_0)}{\sigma^{1/2}}\l(\frac\beta m\r)^{\frac54}\\
+484\sqrt2\tr(\Sigma_0)\l(\frac\beta m\r)^{\frac32} + 2\sqrt2\cdot11^3\|\Sigma_0\|^{\frac12}
\frac{\tr(\Sigma_0)^{3/2}}{\sigma}\l(\frac\beta m\r)^2
+4\cdot11^3\frac{\tr(\Sigma_0)^{3/2}}{\sigma^{1/2}}\l(\frac\beta m \r)^{9/4}.
\end{multline*}
Using the bounds \eqref{support-1} and \eqref{support-2} with some algebraic manipulations, we have the second bound in Lemma \ref{key-sublemma} holds with probability at least $1-2de^{-\beta}$. 
\end{proof}

\subsection{Proof of Lemma \ref{ratio-bound}}
We divide our analysis into the following four cases:
\begin{enumerate}
\item If $\|Z_i\|_2^2\leq1/\theta$ and $\|Z_i-\mu\|_2^2\leq1/\theta$, then, we have $h_\mu(Z_i) = 1$.

\item If $\|Z_i\|_2^2\leq1/\theta$ and $\|Z_i-\mu\|_2^2>1/\theta$. Since $\|\mu\|\leq B_\beta$, it follows
$\|Z_i-\mu\|_2\leq\sqrt{1/\theta}+B_\beta$, and we have
\begin{align*}
h_\mu(Z_i) &= \frac{1/\theta}{\|Z_i-\mu\|_2^2}\leq1,\\
h_\mu(Z_i) &\geq \frac{1/\theta}{\l(\sqrt{1/\theta}+B_\beta\r)^2}
=\frac{1}{1+2B_\beta\sqrt{\theta} + B_\beta^2\theta}\\
&\geq 1 - 2B_\beta\sqrt{\theta} - B_\beta^2\theta,
\end{align*}
where the last inequality follows from the fact $\frac{1}{1+x}\geq1-x,~\forall x\geq0$.

\item If $\|Z_i\|_2^2>1/\theta$ and $\|Z_i-\mu\|_2^2 \leq 1/\theta$. Since $\|\mu\|_2\leq B_\beta$, it follows $\|Z_i\|_2\leq\sqrt{1/\theta}+B_\beta$, and we have
\begin{align*}
h_\mu(Z_i) &= \frac{\|Z_i\|_2^2}{1/\theta}\geq1,\\
h_\mu(Z_i) &\leq \frac{\l(\sqrt{1/\theta}+B_\beta\r)^2}{1/\theta}
=1+2B_\beta\sqrt{\theta} + B_\beta^2\theta.
\end{align*}

\item If $\|Z_i\|_2^2>1/\theta$ and $\|Z_i-\mu\|_2^2 > 1/\theta$. Then, we have
\begin{align*}
h_\mu(Z_i) &= \frac{\|Z_i\|_2^2}{\|Z_i - \mu\|_2^2}\leq \frac{(\|Z_i - \mu\|_2+B_\beta)^2}{\|Z_i - \mu\|_2^2}\\
&\leq\l(\frac{1/\sqrt{\theta}+B_\beta}{1/\sqrt{\theta}}\r)^2\leq  1+2B_\beta\sqrt{\theta} + B_\beta^2\theta,\\
h_\mu(Z_i) &\geq \frac{\|Z_i\|_2^2}{(\|Z_i\|_2 + B_\beta)^2}
\geq\l(\frac{1/\sqrt{\theta}}{1/\sqrt{\theta}+B_\beta}\r)^2\\
&=\frac{1}{1+2B_\beta\sqrt{\theta} + B_\beta^2\theta}
\geq 1 - 2B_\beta\sqrt{\theta} - B_\beta^2\theta,
\end{align*}
\end{enumerate}
Overall, we proved the lemma.

\subsection{Proof of Lemma \ref{bound-on-sigma}}
By definition,
\[
B = \sup_{\|\mathbf{v}\|_2\leq1}\expect{|\langle\mathbf{v},X\rangle|^4}
\geq \expect{\left|X^j\right|^4},~\forall j=1,2,\cdots,d,
\]
where $X^j$ denotes the $j$-th entry of the random vector $X$. Also, for any fixed vector $\mathbf{v}\in\mathbb{R}^d$, we have

\begin{align*}
&0\leq \expect{\left(|\langle\mathbf{v},X\rangle|^2-\left|X^j\right|^2\right)^2}
= \expect{|\langle\mathbf{v},X\rangle|^4} + \expect{\left|X^j\right|^2} 
- 2 \expect{|\langle\mathbf{v},X\rangle|^2\left|X^j\right|^2}\\
&\Rightarrow
\expect{|\langle\mathbf{v},X\rangle|^4} + \expect{\left|X^j\right|^2}\geq
2 \expect{|\langle\mathbf{v},X\rangle|^2\left|X^j\right|^2},~~
\forall j=1,2,\cdots,d.
\end{align*}
Taking the supremum from both sides of the above inequality and use the previous bound on $B$, we get
\[
\sup_{\|\mathbf{v}\|_2\leq1}\expect{|\langle\mathbf{v},X\rangle|^4}
\geq \sup_{\|\mathbf{v}\|_2\leq1}\expect{|\langle\mathbf{v},X\rangle|^2\left|X^j\right|^2},~~\forall j=1,2,\cdots,d.
\]
Summing over $i=1,2,\cdots,d$ gives
\begin{multline*}
Bd = \sup_{\|\mathbf{v}\|_2\leq1}\expect{|\langle\mathbf{v},X\rangle|^4}d
\geq \sum_{j=1}^d\sup_{\|\mathbf{v}\|_2\leq1}\expect{|\langle\mathbf{v},X\rangle|^2\left|X^j\right|^2}
\geq \sup_{\|\mathbf{v}\|_2\leq1}\expect{|\langle\mathbf{v},X\rangle|^2\left\|X\right\|^2}\\
=\left\| XX^T\|X\|_2^2 \right\| = \sigma_0^2.
\end{multline*}

\subsection{Proof of Lemma \ref{effective-rank-bound}}
First of all, let $Z = X-\mu_0$, then, we have $\expect{Z} = 0$. The lower bound of $\sigma_0^2$ follows directly from Corollary \ref{FKG-bound}. It remains to show the upper bound.
Note that by Cauchy-Schwarz inequality, 
\begin{align*}
\sigma_0^2=\left\| ZZ^T\|Z\|_2^2 \right\| 
=&   \sup_{\|\mathbf{v}\|_2\leq1}\expect{\langle Z,\mathbf{v}\rangle^2\|Z\|_2^2}\\
\leq& \sup_{\|\mathbf{v}\|_2\leq1}\expect{\langle Z,\mathbf{v}\rangle^4}^{1/2}
\expect{\| Z \|_2^4}^{1/2}.
\end{align*}
We then bound the two terms separately. For any vector $\mathbf{x}\in\mathbb{R}^d$, let $x^j$ be the $j$-th entry. Note that for any $\mathbf{v}\in \mathbb{R}^d$ such that $\|\mathbf{v}\|_2\leq1$, we have
\begin{align*}
\expect{\langle Z,\mathbf{v}\rangle^4}^{1/2}
\leq R\cdot \expect{\langle Z,\mathbf{v}\rangle^2} \leq R \sup_{\|\mathbf{v}\|_2\leq1}\expect{\langle Z,\mathbf{v}\rangle^2}\leq R  \|\Sigma_0\|,
\end{align*}
where the first inequality uses the fact that the kurtosis is bounded.

Also, we have
\begin{align*}
\expect{\|Z\|_2^4}^{1/2} =&\l( \sum_{j=1}^d\expect{(Z^j)^4} + \sum_{j,k=1,~j\neq k}^d\expect{(Z^j)^2(Z^k)^2}\r)^{1/2}\\
\leq& \l( \sum_{j=1}^d\expect{(Z^j)^4} +  \sum_{j,k=1,~j\neq k}^d \expect{(Z^j)^4}^{1/2}\expect{(Z^k)^4}^{1/2} \r)^{1/2}\\
\leq& \sum_{j=1}^d\sqrt{\expect{(Z^j)^4}}\leq  R\cdot\sum_{j=1}^d\expect{(Z^j)^2}
=  R\cdot\tr(\Sigma_0)
\end{align*}
Combining the above two bounds gives 
\[
\sigma_0^2\leq  R^2\|\Sigma_0\|\tr(\Sigma_0),\]
which implies the result.

\end{document}